\documentclass[12pt,twoside,a4paper]{article}

\title{Stochastic Proximal Gradient Methods for Nonconvex Problems in Hilbert Spaces}
\author{Caroline Geiersbach\thanks{Weierstrass Institute, Mohrenstrasse 39, 10117, Berlin, Germany 
  (\texttt{caroline.geiersbach@wias-berlin.de}).}
\and Teresa Scarinci\thanks{Department of Information Engineering, Computer Science and Mathematics, University of L’Aquila, Via Vetoio - Loc. Coppito, 67010, L’Aquila, Italy
  (\texttt{teresa.scarinci@univaq.it})}}
  \usepackage{amsmath, amsfonts}
 \usepackage{amssymb}       
\usepackage{mathtools}
\usepackage[utf8]{inputenc}
\usepackage{bbm, bm}
 \usepackage{caption}
 \usepackage{algorithm}
\usepackage{algorithmic}
\usepackage{graphicx}
\usepackage{varwidth}
\usepackage{amsthm}

\DeclareMathOperator*{\argmin}{arg\,min}
\DeclareMathOperator*{\esssup}{ess\,sup}

\DeclareMathOperator*{\dom}{dom}

\newcommand{\R}{\mathbb{R}}
\newcommand{\N}{\mathbb{N}}

\newcommand{\E}{\mathbb{E}}
\newcommand{\Exi}{\mathbb{E}_{\xi}}
\newcommand{\pP}{\mathbb{P}}

\newcommand{\D}{\,\mathrm{d}}
\newcommand{\prox}{\textup{prox}}

\numberwithin{equation}{section}

\newtheorem{theorem}{Theorem}[section]
\newtheorem{lemma}[theorem]{Lemma}

\newtheorem{proposition}[theorem]{Proposition}

\theoremstyle{remark}
\newtheorem{remark}[theorem]{Remark}

\theoremstyle{definition}

\newtheorem{definition}[theorem]{Definition}

\newcounter{assumption}
\newtheorem{assumption}[theorem]{Assumption}
\newcounter{subassumption}[assumption]
\renewcommand{\thesubassumption}{(\textit{\roman{subassumption}})}
\makeatletter
\renewcommand{\p@subassumption}{\theassumption}
\makeatother
\newcommand{\subasu}{
  \refstepcounter{subassumption}%
  \thesubassumption~\ignorespaces}

\begin{document}
\maketitle
\begin{abstract}
For finite-dimensional problems, stochastic approximation methods have long been used to solve stochastic optimization problems. Their application to infinite-dimensional problems is less understood, particularly for nonconvex objectives. This paper presents convergence results for the stochastic proximal gradient method applied to Hilbert spaces, motivated by optimization problems with partial differential equation (PDE) constraints with random inputs and coefficients. We study stochastic algorithms for nonconvex and nonsmooth problems, where the nonsmooth part is convex and the nonconvex part is the expectation, which is assumed to have a Lipschitz continuous gradient. The optimization variable is an element of a Hilbert space. We show almost sure convergence of strong limit points of the random sequence generated by the algorithm to stationary points. We demonstrate the stochastic proximal gradient algorithm on a tracking-type functional with a $L^1$-penalty term constrained by a semilinear PDE and box constraints, where input terms and coefficients are subject to uncertainty. We verify conditions for ensuring convergence of the algorithm and show a simulation.
\end{abstract}

\section{Introduction}
In this paper, we focus on stochastic approximation methods for solving a stochastic optimization problem on a Hilbert space $H$ of the form
 \begin{equation} \label{eq:ProblemFormulation-basic} \tag{P}
 \min_{u \in H} \{f(u) = j(u) + h(u) \},
\end{equation}
where the expectation $j(u) =  \E [J(u, \xi)]$ is generally nonconvex with a Lipschitz continuous gradient and $h$ is a proper, lower semicontinuous, and convex function that is generally nonsmooth. 

Our work is motivated by applications to PDE-constrained optimization under uncertainty, where a nonlinear PDE constraint can lead to an objective function that is nonconvex with respect to the Hilbert-valued variable. To handle the (potentially infinite-dimensional) expectation, algorithmic approaches for solving such problems involve either some discretization of the stochastic space or an ensemble-based approach with sampling or carefully chosen quadrature points. Stochastic discretization includes polynomial chaos and the stochastic Galerkin method; cf.~\cite{Keshavarzzadeh2017,Kunoth2016,Lee2013,Rosseel2012}. For ensemble-based methods, the simplest method is sample average approximation (SAA), where the original problem is replaced by a proxy problem with a fixed set of samples, which can then be solved using a deterministic solver. A number of standard improvements to Monte Carlo sampling have been applied to optimal control problems in, e.g.,~\cite{Ali2017,VanBarel2017}. Another ensemble-based approach is the stochastic collocation method, which has been used in optimal control problems in e.g.~\cite{Rosseel2012,Tiesler2012}. Sparse-tensor discretization has been used for optimal control problems in, for instance, \cite{Kouri2013,Kouri2014a}.

The approach we use is an ensemble-based approach called stochastic approximation, which is fundamentally different in the sense that sampling takes place dynamically as part of the optimization procedure, leading to an algorithm with low complexity and computational effort when compared to other approaches. Stochastic approximation originated in a groundbreaking paper by \cite{Robbins1951}, where an iterative method to find the root of an unknown function using noisy estimates was proposed. The authors of \cite{Kiefer1952} used this idea to solve a regression problem using finite differences subject to noise. Algorithms of this kind, with bias in addition to stochastic noise, are sometimes called stochastic quasi-gradient methods; see, e.g., \cite{Ermoliev1969,Uryasev1992}.Basic versions of these algorithms rely on positive step sizes $t_n$ of the form $\sum_{n=1}^\infty t_n = \infty$ and $\sum_{n=1}^\infty t_n ^2 < \infty$. The (almost sure) asymptotic convergence of stochastic approximation algorithms for convex problems is classical in finite dimensions; we refer to the texts by \cite{Duflo2013,Kushner1978}. 

There have been a number of contributions with proofs of convergence of the stochastic gradient method for unconstrained nonconvex problems; see \cite{Bottou1998,Bottou2018,Shapiro1996,Wardi1989}. Fewer results exist for constrained and/or nonsmooth nonconvex problems. A randomized stochastic algorithm was proposed by \cite{Ghadimi2016}; this scheme involves running a stochastic approximation process and randomly choosing an iterate from the generated sequence. There have been some contributions involving constant step sizes with increasing sampling; see \cite{Lei2018,Reddi2016a}. Convergence of projection-type methods for nonconvex problems was shown in \cite{Kushner2003} and for prox-type methods by \cite{Davis2018}.

As far as stochastic approximation on function spaces is concerned, many contributions were motivated by applications with nonparametric statistics. Perhaps the oldest example is from \cite{Venter1966}. Goldstein \cite{Goldstein1988} studied an infinite-dimensional version of the Kiefer--Wolfowitz procedure. A significant contribution for unconstrained problems was by \cite{Yin1990}. Projection-type methods were studied by \cite{Barty2007,Chen2002,Culioli1990,Nixdorf1984}.

In this paper, we prove convergence results for nonconvex and nonsmooth problems in Hilbert spaces. We present convergence analysis that is based on the recent contributions in \cite{Davis2018,Lei2018}. Applications of the stochastic gradient method to PDE-constrained optimization have already been explored by \cite{Geiersbach2019,Martin2018}. In these works, however, convexity of the objective function is assumed, leaving the question of convergence in the more general case entirely open. We close that gap by making the following contributions:
\begin{itemize}
 \item For an objective function that is the sum of a smooth, generally nonconvex expectation and a convex, nonsmooth term, we prove that strong accumulation points of iterates generated by the method are stationary points.
 \item We show that convergence holds even in the presence of systematic additive bias, which is relevant for the application in mind.
 \item We demonstrate the method on an application to PDE-constrained optimization under uncertainty and verify conditions for convergence.
\end{itemize}

The paper is organized as follows. In Sect.~\ref{sec:background}, notation and background is given. Convergence of two related algorithms is proven in Sect.~\ref{sec:convergence}. In Sect.~\ref{sec:numerical-experiments}, we introduce a problem in PDE-constrained optimization under uncertainty, where coefficients in the semilinear PDE constraint are subject to uncertainty. The problem is shown to satisfy conditions for convergence, and numerical experiments demonstrate the method. We finish the paper with closing remarks in Sect.~\ref{sec:conclusion}.

\section{Notation and Background}
\label{sec:background}
We recall some notation and background from convex analysis and stochastic processes; see \cite{Bauschke2011,Clarke1990,Metivier2011,Pisier2016}.

Let $H$ be a Hilbert space with the scalar product $\langle \cdot, \cdot \rangle$ and norm $\lVert \cdot \rVert$. The symbols $\rightarrow$ and $\rightharpoonup$ denote strong and weak convergence, respectively. The set of proper, convex, and lower semicontinuous functions $h:H \rightarrow (-\infty, \infty]$ is denoted by $\Gamma_0(H).$ Given a function $h \in \Gamma_0(H)$ and $t > 0$, the proximity operator $\prox_{th}:H \rightarrow H$ is given by
\begin{equation*}
 \label{eq:prox-definition}
 \prox_{t h}(u) := \argmin_{v \in H} \left( h(v) + \frac{1}{2t} \lVert v-u\rVert^2\right).
\end{equation*}
We recall that for a proper function $h:H \rightarrow (-\infty, \infty]$, the subdifferential (in the sense of convex analysis) is the set-valued operator
$$\partial h: H \rightrightarrows H: u \mapsto \{ v \in H: \langle y - u, v \rangle + h(u) \leq h(y) \quad  \forall y \in H \}.$$ 
For any $h \in \Gamma_0(H)$, the subdifferential $\partial h$ is maximally monotone. 
The domain of $h$ is denoted by $\dom (h)$. The indicator function of a set $C$ is denoted by $\delta_C$, where $\delta_C(u) = 0$ if $u \in C$ and $\delta_C(u) = \infty$ otherwise. The sum of two sets $A$ and $B$ with $\lambda \in \R$ is given by {$A+\lambda B:=\{ a+\lambda b: a \in A, b \in B\}.$}  The distance of a point $u$ to a nonempty, closed set $A$ is denoted by $d(u,A):=\inf_{a \in A} \lVert u-a\rVert$ and the diameter of $A$ is denoted by the symbol $\text{diam}(A):=\sup_{u, v \in A} \lVert u - v \rVert.$ For a nonempty and convex set $C$, the normal cone $N_C(u)$ at $u \in C$ is defined by
$$N_C(u):= \{ z \in H: \langle z, w-u \rangle \leq 0, \quad \forall w \in C\}.$$
We set $N_C(u) := \emptyset$ if $u \notin C$. We recall that $\partial \delta_C(u) = N_C(u)$ for all $u \in C$. 
If $h_1,h_2 \in \Gamma_0(H)$ and $\textup{dom} (h_2) = H$, then $\partial [h_1(u) + h_2(u)] = \partial h_1(u) +\partial h_2(u)$.
If $h$ is proper and $u \in \textup{dom}(h)$, then $\partial h(u)$ is closed and convex. 
We recall that the graph of $ \partial h$ for a function $h \in \Gamma_0(H)$, given by the set $\textup{gra}(\partial h) = \{ (u,\partial h(u)): u \in H\}$, is sequentially closed in the strong-to-weak topology, meaning that for $u_n \rightarrow u$, $\zeta_n \in \partial h(u_n)$, and $\zeta_n \rightharpoonup \zeta$, it follows that $\zeta \in \partial h(u)$. 
The normal cone $N_C(u)$ is strong-to-weak sequentially closed if $C$ is convex.

Throughout, $(\Omega, \mathcal{F}, \pP)$ will denote a probability space, where $\Omega$ represents the sample space, $\mathcal{F} \subset 2^{\Omega}$ is the $\sigma$-algebra of events on the power set of $\Omega$, denoted by $2^{\Omega}$, and $\pP\colon \Omega \rightarrow [0,1]$ is a probability measure. Given a random vector $\xi:\Omega \rightarrow \Xi \subset \R^m$, we write $\xi \in \Xi$ to denote a  realization of the random vector. The operator $\E[\cdot]$ denotes the expectation with respect to this distribution; for a parametrized functional $J: H \times \Xi \rightarrow \R$, this is defined as the integral over all elements in $\Omega$, i.e.,
$$\E[J(u,\xi)] = \int_\Omega J(u,\xi(\omega)) \D \pP(\omega).$$
A filtration is a sequence $\{ \mathcal{F}_n\}$ of sub-$\sigma$-algebras of $\mathcal{F}$ such that {$\mathcal{F}_1 \subset \mathcal{F}_2 \subset \cdots \subset \mathcal{F}.$} 
We define a discrete $H$-valued stochastic process as a collection of $H$-valued random variables indexed by $n$, in other words, the set $\{ \beta_n: \Omega \rightarrow H \, \vert \, n \in \N\}.$ 
The  stochastic process is said to be adapted to a filtration $\{ \mathcal{F}_n \}$ if and only if $\beta_n$ is $\mathcal{F}_n$-measurable for all $n$. The natural filtration is the filtration generated by the sequence $\{\beta_n\}$ and is given by $\mathcal{F}_n = \sigma(\{\beta_1, \dots ,\beta_n\})$.\footnote{The $\sigma$-algebra generated by a random variable $\beta:\Omega \rightarrow \R$ is given by $\sigma(\beta) = \{ \beta^{-1}(B): B \in \mathcal{B}\}$, where $\mathcal{B}$ is the Borel $\sigma$-algebra on $\R$. Analogously, the $\sigma$-algebra generated by the set of random variables $\{ \beta_1, \dots, \beta_n\}$ is the smallest $\sigma$-algebra such that $\beta_i$ is measurable for all $i=1, \dots, n$.} If for an event $F \in \mathcal{F}$ it holds that $\pP(F) = 1$, or equivalently, $\pP(\Omega\backslash F) = 0$, we say $F$ occurs almost surely (a.s.). Sometimes we also say that such an event occurs with probability one. A sequence of random variables $\{\beta_n\}$ is said to converge almost surely to a random variable $\beta$ if and only if 
$$\pP\left(\left\lbrace \omega \in \Omega: \lim_{n \rightarrow \infty} \beta_n(\omega) = \beta(\omega) \right\rbrace\right) = 1.$$
For an integrable random variable $\beta:\Omega \rightarrow \R$, the conditional expectation is denoted by $\E[\beta | \mathcal{F}_n]$, which is itself a random variable that is $\mathcal{F}_n$-measurable and which satisfies $\int_A \E[\beta | \mathcal{F}_n](\omega) \D \pP(\omega) = \int_A \beta(\omega) \D \pP(\omega)$ for all $A \in \mathcal{F}_n$. Almost sure convergence of $H$-valued stochastic processes and conditional expectation are defined analogously.  

Given a random operator $F:X \times \Omega \rightarrow Y$,  where $X$ and $Y$ are Banach spaces, we will sometimes use the notation $F_\omega:=F(\cdot, \omega):X \rightarrow Y$ for a fixed (but arbitrary) $\omega \in \Omega$. For a Banach space $(X,\lVert \cdot \rVert_X)$, the Bochner space $L^p(\Omega,X)$ is the set of all (equivalence classes of) strongly measurable functions $u:\Omega \rightarrow X$ having finite norm, where the norm is defined by
$$\lVert u \rVert_{L^p(\Omega,X)}:= \begin{cases}
                                     (\int_\Omega \lVert u(\omega) \rVert_X^p \D \pP(\omega))^{1/p}, \quad &p < \infty\\
                                     \esssup_{\omega \in \Omega} \lVert u(\omega) \rVert_X, \quad &p=\infty
                                    \end{cases}.
$$
A sequence $\{\beta_n\}$ in $L^1(\Omega, X)$ is called a martingale if a filtration $\{ \mathcal{F}_n\}$ exists such that $\beta_n$ is $\mathcal{F}_n$-measurable and $\E[\beta_{n+1}|\mathcal{F}_n] = \beta_{n}$ is satisfied for all $n$.

For an open subset $U$ of a Banach space $X$ and a function $J_\omega:U \rightarrow \R$, we denote the G\^{a}teaux derivative at $u \in U$ in the direction $v \in X$  by $dJ_\omega(u; v).$ The Fr\'echet derivative at $u$ is denoted by $J_\omega':U \rightarrow \mathcal{L}(X,\R)$, where $\mathcal{L}(X,\R)$ is the set of bounded and linear operators mapping $X$ to $\R$.  We recall this is none other than the dual space $X^*$ and we denote the dual pairing by $\langle \cdot, \cdot \rangle_{X^*,X}$. For an open subset $U$ of a Hilbert space $H$ and a Fr\'echet differentiable function $j:U \rightarrow \R$, the gradient $\nabla j:U \rightarrow H$ is the Riesz representation of $j':U \rightarrow H^*$, i.e.,~it satisfies $\langle \nabla j(u), v \rangle = \langle j'(u),v \rangle_{H^*,H}$ for all $u \in U$ and $v \in H.$  
In Hilbert spaces, the Riesz representation relates elements of the dual space to the Hilbert space itself, allowing us to drop the dual pairing notation and use simply $\langle \cdot, \cdot \rangle$. 

The notation $C_L^{1,1}(U)$ is used to denote the set of continuously differentiable functions on $U \subset H$ with an $L$-Lipschitz gradient, meaning
$\lVert \nabla j(u) - \nabla j(v) \rVert \leq L \lVert u - v \rVert$
is satisfied for all $u,v \in U.$
The following lemma gives a classical Taylor estimate for such functions.
\begin{lemma}\label{lemma:lipschitzderivative-Hilbert}
Suppose $j \in C_L^{1,1}(U)$, $U\subset H$ open and convex. Then for all $u, v \in U$,
$$j(v) + \langle \nabla j(v), u-v\rangle - \frac{L}{2} \lVert u - v \rVert^2 \leq 
j(u) \leq  j(v) + \langle \nabla j(v), u-v\rangle + \frac{L}{2} \lVert u-v\rVert^{2}.$$
\end{lemma}

\section{Asymptotic Convergence Results}
\label{sec:convergence}
In this section, we show asymptotic convergence results for two variants of the stochastic proximal gradient method in Hilbert spaces for solving Problem \eqref{eq:ProblemFormulation-basic}. Let $G:H \times \Xi \rightarrow H$ be a parametrized operator (the \textit{stochastic gradient}) approximating (in a sense to be specified later) the gradient $\nabla j:H \rightarrow H$ and let $t_n$ be a positive step size. Both algorithms in this section will share the basic iterative form
$$u_{n+1} := \prox_{t_n h}(u_n - t_n G(u_n,\xi_n)),$$
where $h$ is the nonsmooth term from Problem \eqref{eq:ProblemFormulation-basic}. The following assumptions will be in force in all sections.

\begin{assumption}
\label{asu1}
Let $\{\mathcal{F}_n \}$ be a filtration and let $\{ u_n\}$ and $\{G(u_n,\xi_n)\}$ be sequences of iterates and stochastic gradients. We assume \\
\subasu \label{asu1i} The sequence $\{ u_n\}$ is a.s.~contained in a bounded set $V \subset H$ and $u_n$ is adapted to $\mathcal{F}_n$ for all $n$.\\
\subasu \label{asu1ii} On an open and convex set $U$ such that $V \subset U  \subset H$, the expectation $j\in C_L^{1,1}(U)$ is bounded below.\\
\subasu \label{asu1iii} For all $n$, the $H$-valued random variable
$r_n := \E[G(u_n,\xi_n) | \mathcal{F}_n] - \nabla j(u_n)$
is adapted to $ \mathcal{F}_n$ and for $K_n:=\esssup_{\omega \in \Omega} \lVert r_n(\omega)\rVert$, {$ \sum_{n=1}^\infty t_n K_n< \infty$} and $\sup_{n} K_n<\infty$ are satisfied. \\
\subasu \label{asu1iv} For all $n$, 
$\mathfrak{w}_n := G(u_n, \xi_n) - \E[G(u_n, \xi_n) | \mathcal{F}_n]$
is an $H$-valued random variable.
\end{assumption}

\begin{remark}
\label{remark:general-assumptions}
The assumption that the sequence $\{ u_n\}$ stays bounded with probability one is by no means automatically fulfilled, but can be verified or enforced in different ways. We refer to \cite[Section 5.2]{Bottou1998} and \cite[Section 6.1]{Davis2018} for conditions on the function, constraint set, and/or regularizers that ensure boundedness of iterates. The conditions in Assumption~\ref{asu1} allow for additive bias $r_n$ in the stochastic gradient in addition to zero-mean error $\mathfrak{w}_n$. The requirement that $u_n$ and $r_n$ are adapted to $\mathcal{F}_n$ is automatically fulfilled if $\{ \mathcal{F}_n\}$ is chosen to be the natural filtration generated by $\{\xi_1, \dots, \xi_n \}$. Together, Assumption~\ref{asu1iii} and Assumption~\ref{asu1iv} imply
$$G(u_n,\xi_n) = \nabla j(u_n) + r_n + \mathfrak{w}_n$$
and $\E[\mathfrak{w}_n | \mathcal{F}_n] = 0.$ Notice that a single realization $\xi_n \in \Xi$ can be replaced by $m_n$ independently drawn realizations $\xi_{n}^1, \dots, \xi_{n}^{m_n} \in \Xi$ since
\begin{equation*}
\E[G(u_n,\xi_n)|\mathcal{F}_n] =\frac{1}{m_n} \E \left[\sum_{i=1}^{m_n} G(u_n,\xi_{n}^{i}) |\mathcal{F}_n \right].
\end{equation*}
This set of $m_n$ samples is sometimes called a ``batch''; batches clearly reduce the variance of the stochastic gradient. 
\end{remark}

The result in Sect.~\ref{subsection:SPGM-Variance-Reduced} shows asymptotic convergence of the proximal gradient method with constant step sizes and increasing sampling. In Sect.~\ref{subsection:ODE-proof}, we switch to the versatile ordinary differential equation (ODE) method to prove convergence of the stochastic proximal gradient method with decreasing step sizes. We emphasize that the convergence results generalize existing convergence theory from the finite-dimensional case. Our analysis includes convergence in possibly infinite-dimensional Hilbert spaces. Additionally, we allow for stochastic gradients subject to additive bias, which is not covered by existing results. This theory can be used to develop mesh refinement strategies in applications with PDEs \cite{Geiersbach2020b}.
 
\subsection{Variance-Reduced Stochastic Proximal Gradient Method}
\label{subsection:SPGM-Variance-Reduced}
In this section, we show under what conditions the variance-reduced stochastic proximal gradient method converges to stationary points for Problem \eqref{eq:ProblemFormulation-basic}. With $\xi_n = (\xi_{n}^1,\dots, \xi_{n}^{m_n}),$ the stochastic gradient is given by the average 
$$G(u_n,\xi_n) = \frac{\sum_{i=1}^{m_n}G(u_n,\xi_{n}^i)}{m_n}$$
over an \textit{increasing} number of samples $m_n$. The algorithm is presented below, which uses constant step sizes $t_n \equiv t$ depending on the Lipschitz constant $L$ from Assumption~\ref{asu1ii}.
\begin{algorithm}[H] 
\begin{algorithmic}[0] 
\STATE \textbf{Initialization:} $u_1 \in H$, $0<t<\tfrac{1}{2L}$ 
\FOR{$n=1,2,\dots$}
\STATE Generate independent $\xi_{n}^1, \dots, \xi_{n}^{m_n} \in \Xi$, independent of $\xi_{1}^1, \dots, \xi_{n-1}^{m_{n-1}}$
\STATE $u_{n+1}:=\prox_{t h}\left( u_n - t \frac{\sum_{i=1}^{m_n}G(u_n,\xi_{n}^i)}{m_n}\right)$
\ENDFOR
\end{algorithmic}
\captionof{algorithm}{Variance-Reduced Stochastic Proximal Gradient Method} 
\label{alg:PSG_Hilbert_Nonconvex}
\end{algorithm}

\begin{remark}
\label{remark:indicator-projection}
If $h(u)=\delta_C(u)$ and $\pi_C$ denotes the projection onto $C$, then the algorithm reduces to  
$u_{n+1}:=\pi_C\left( u_n - t \frac{\sum_{i=1}^{m_n}G(u_n,\xi_{n}^i)}{m_n}\right),$
i.e., the variance-reduced projected stochastic gradient method. 
\end{remark}

In addition to Assumption~\ref{asu1}, the following assumptions will be in force in this section.

\begin{assumption}
\label{assumption:well-posedness-constrained}
Let $\{ u_n\}$ and $\{G(u_n,\xi_n)\}$ be generated by  Algorithm~\ref{alg:PSG_Hilbert_Nonconvex}. We assume\\ \setcounter{subassumption}{0}
\subasu \label{asu3i} The function $h$ satisfies $h \in \Gamma_0(H)$.\\ 
\subasu \label{asu3ii} For all $n$, 
$$w_n : = \frac{\sum_{i=1}^{m_n}G(u_n,\xi_{n}^i)}{m_n} - \nabla j(u_n)$$
an $H$-valued random variable and there exists an $M \geq 0$ such that $\E[\lVert w_n\rVert^2 | \mathcal{F}_n] \leq \frac{M}{m_n}$ and $\sum_{n=1}^\infty \tfrac{1}{m_n} < \infty$. 
\end{assumption}

\begin{remark}
We use assumptions similar to those found in \cite{Lei2018}, but we do not require  the effective domain of $h$ to be bounded; we instead use boundedness of the iterates by Assumption~\ref{asu1i}. Notice that $w_n = r_n + \mathfrak{w}_n$ from Assumption~\ref{asu1iv}, hence Assumption~\ref{asu3ii} also provides a condition on the rate at which $r_n$ and $\mathfrak{w}_n$ must decay. 
\end{remark}

For the convergence result, we need the following lemma \cite{Robbins1971}.
\begin{lemma}[Robbins--Siegmund] \label{lemma:Robbins-Siegmund-Chapter-3}
Assume that $\{\mathcal{F}_n\}$ is a filtration and 
$v_n$, $a_n$, $b_n$, $c_n$ nonnegative random variables adapted to $\mathcal{F}_n.$ If
\begin{equation*}\label{eq:Robbins-Siegmund}
\E[v_{n+1} | \mathcal{F}_n] \leq v_n(1+a_n)+ b_n-c_n \quad \text{a.s.}
\end{equation*}
and $\sum_{n=1}^\infty a_n < \infty,  \sum_{n=1}^\infty b_n < \infty$ a.s., then with probability one, $\{v_n\}$ is convergent and $\sum_{n=1}^\infty c_n < \infty$. 
\end{lemma}

To show convergence, we first present a technical lemma.

\begin{lemma}
\label{lemma:fundamental-inequality-prox}
Let $u\in U$ and $t>0$. Suppose $v:=\prox_{t h}(u-t g) \in U$ for a given $g \in H$. Then for any $z \in U$, 
\begin{equation}
\label{eq:fundamental-inequality-prox}
\begin{aligned}
 f(v) &\leq f(z) + \langle v-z, \nabla j(u) - g\rangle + \left( \frac{L}{2}- \frac{1}{2t}\right) \lVert v-u \rVert^2\\
 &\quad\quad + \left( \frac{L}{2} + \frac{1}{2t}\right) \lVert z-u \rVert^2 - \frac{1}{2t} \lVert v-z \rVert^2.
\end{aligned} 
\end{equation}
\end{lemma}

\begin{proof}
We first claim that for all $y,z\in H$, $t>0$ and $p=\prox_{th}(y)$,
\begin{equation}
\label{eq:prox-inequality}
h(p) + \frac{1}{2t} \lVert p-y \rVert^2 \leq h(z) + \frac{1}{2t} \lVert z-y \rVert^2 - \frac{1}{2t} \lVert p-z \rVert^2.
\end{equation}
This follows by definition of the $\prox$ operator. 
Indeed, for $t > 0$, $p = \prox_{t h}(y)$ if and only if for all $z \in H$,
\begin{equation}
\label{eq:first-inequality-prox}
h(z) \geq h(p) + \frac{1}{t} \langle y -p, z-p \rangle.
\end{equation}
It is straightforward to verify the following equality (the law of cosines)
\begin{equation}
\label{eq:second-inequality-cosine-law}
\lVert z - y \rVert^2 = \lVert z-p\rVert^2 + \lVert p-y\rVert^2 - 2 \langle y-p, z-p\rangle. 
\end{equation}
Multiplying \eqref{eq:second-inequality-cosine-law} by $\tfrac{1}{2t}$ and adding it to \eqref{eq:first-inequality-prox}, we get \eqref{eq:prox-inequality}. Now, since $j\in C^{1,1}_L(U)$, it follows by Lemma~\ref{lemma:lipschitzderivative-Hilbert} for $u,v,z\in U$ that
\begin{align}
 j(v)& \leq j(u) + \langle \nabla j(u), v-u \rangle + \frac{L}{2} \lVert v - u \rVert^2, \label{eq:Lipschitz-inequality1}\\
  j(u)& \leq j(z) + \langle \nabla j(u), u-z \rangle + \frac{L}{2} \lVert z - u \rVert^2. \label{eq:Lipschitz-inequality2}
\end{align}
Combining \eqref{eq:Lipschitz-inequality1} and \eqref{eq:Lipschitz-inequality2}, we get
\begin{equation}
 \label{eq:Lipschitz-inequality-combined}
 j(v) \leq j(z) + \langle \nabla j(u), v-z \rangle + \frac{L}{2} \lVert v - u \rVert^2  + \frac{L}{2} \lVert z - u \rVert^2.
\end{equation}
Now, by \eqref{eq:prox-inequality} applied to $v = \prox_{th}(u-tg)$,
\begin{align*}
 h(v) + \frac{1}{2t} \lVert v-(u-tg)\rVert^2 &\leq h(z) +\frac{1}{2t} \lVert z-(u-tg)\rVert^2 -\frac{1}{2t} \lVert v-z\rVert^2 
 \end{align*}
 if and only if
 \begin{equation}
  \label{eq:prox-inequality-with-values}
 \begin{aligned}
& h(v)  + \frac{1}{2t} \lVert v-u \rVert^2 + \langle v-u,g\rangle\\
 &\quad \quad \leq h(z)+\frac{1}{2t} \lVert z-u\rVert^2+\langle z-u,g\rangle -\frac{1}{2t} \lVert v-z\rVert^2. 
\end{aligned}
\end{equation}
Finally, adding \eqref{eq:Lipschitz-inequality-combined} and \eqref{eq:prox-inequality-with-values}, and using that $f = j+h$, we get \eqref{eq:fundamental-inequality-prox}.
\end{proof}

In the following, we define
\begin{equation}
\label{eq:prox-point-full-gradient}
\bar{u}_{n+1} := \prox_{t h}(u_n - t \nabla j(u_n) )
\end{equation}
as the iterate at $n+1$ if the true gradient were used. 

\begin{lemma}
\label{lemma:martingale-inequality-fundamental-constant-steps}
For all $n$,
\begin{equation}
\label{inequality-prox-sequence}
 \E[f(u_{n+1}) |\mathcal{F}_n] \leq f(u_n) - \left( \frac{1}{2t} - L \right) \lVert \bar{u}_{n+1} - u_n \rVert^2 + \frac{t}{2} \E[ \lVert w_n \rVert^2 | \mathcal{F}_n] \quad \textup{a.s}. 
\end{equation}
\end{lemma}
\begin{proof}
Using Lemma~\ref{lemma:fundamental-inequality-prox} with $v = \bar{u}_{n+1}$, $u=z=u_n$, and $g = \nabla j(u_n)$, we have 
\begin{equation}
\label{eq:martingale-inequality-proof-inequality1}
f(\bar{u}_{n+1}) \leq f(u_n) + \left( \frac{L}{2} - \frac{1}{t}\right) \lVert \bar{u}_{n+1} - u_n \rVert^2.
\end{equation}
Again using Lemma~\ref{lemma:fundamental-inequality-prox}, with $v=u_{n+1}$, $z=\bar{u}_{n+1}$, $u=u_n$, and $g = \nabla j(u_n) + w_n$, we get
\begin{equation}
\label{eq:martingale-inequality-proof-inequality2}
\begin{aligned}
f(u_{n+1}) &\leq f(\bar{u}_{n+1}) - \langle u_{n+1}-\bar{u}_{n+1}, w_n\rangle + \left( \frac{L}{2} - \frac{1}{2t}\right) \lVert u_{n+1} - u_n \rVert^2 \\
& \qquad + \left( \frac{L}{2} + \frac{1}{2t}\right) \lVert \bar{u}_{n+1} - u_n \rVert^2 - \frac{1}{2t} \lVert  u_{n+1}-\bar{u}_{n+1}\rVert^2.
\end{aligned}
\end{equation}
By Young's inequality, 
$ \langle u_{n+1}-\bar{u}_{n+1}, w_n\rangle \leq \frac{1}{2t}\lVert u_{n+1}-\bar{u}_{n+1} \rVert^2 + \frac{t}{2} \lVert w_n \rVert^2, $
so combining \eqref{eq:martingale-inequality-proof-inequality1} and \eqref{eq:martingale-inequality-proof-inequality2}, we obtain since $0 < t < \tfrac{1}{2L}$ that
\begin{equation}
\begin{aligned}
 \label{eq:martingale-inequality-proof-inequality3}
f(u_{n+1}) &\leq f(u_n) + \left( L-\frac{1}{2t} \right) \lVert  \bar{u}_{n+1}-u_{n}\rVert^2 + \left( \frac{L}{2}-\frac{1}{2t} \right) \lVert  u_{n+1}-u_{n}\rVert^2 \\
& \quad \quad + \frac{t}{2} \lVert w_n \rVert^2 \\
& \leq f(u_n) + \left( L-\frac{1}{2t} \right) \lVert  \bar{u}_{n+1}-u_{n}\rVert^2 + \frac{t}{2} \lVert w_n \rVert^2.
\end{aligned}
\end{equation}

Taking conditional expectation on both sides of \eqref{eq:martingale-inequality-proof-inequality3}, and noting that $\bar{u}_{n+1}$ is $\mathcal{F}_n$-measurable by $\mathcal{F}_n$-measurability of $u_n$, we get \eqref{inequality-prox-sequence}.
\end{proof}

\begin{remark}
Any bounded sequence $\{ u_n\}$ in $H$ contains a weakly convergent subsequence $\{ u_{n_k}\}$ such that $u_{n_k} \rightharpoonup u$ for a $u \in H.$ Generally this convergence is not strong, so we cannot conclude from 
$\lVert  \bar{u}_{n+1}-u_{n}\rVert^2 \rightarrow 0$
that there exists a $\tilde{u}$ such that, for a subsequence $\{ u_{n_k}\}$, {$\lim_{k \rightarrow \infty} \bar{u}_{n_k+1} = \lim_{k \rightarrow \infty} u_{n_k} = \tilde{u}.$} Therefore, to obtain convergence to stationary points, we will assume that $\{ u_n\}$ has a strongly convergent subsequence.
\end{remark}

We are ready to state the convergence result for sequences generated by~Algorithm~\ref{alg:PSG_Hilbert_Nonconvex}.
\begin{theorem}
 \label{theorem:convergence-variance-reduced-stochastic-gradient}
Let Assumption~\ref{asu1} and Assumption~\ref{assumption:well-posedness-constrained} hold. Then
\begin{enumerate}
 \item The sequence $\{ f(u_n)\}$ converges a.s.
  \item The sequence $\{ \lVert \bar{u}_{n+1} - u_n \rVert \}$ converges to zero a.s.
 \item Every strong accumulation point of $\{ u_n\}$ is a stationary point with probability one. 
\end{enumerate}
\end{theorem}

\begin{proof}
The sequence $\{ u_n\}$ is contained in a bounded set $V$ by Assumption~\ref{asu1i}. By Assumption~\ref{asu3i}, $h \in \Gamma_0(H)$ must therefore be bounded below on $V$ \cite[Corollary 9.20]{Bauschke2011}; $j$ is bounded below by Assumption~\ref{asu1ii}. W.l.o.g.~we can thus assume $f \geq 0$. Since $\frac{1}{2t} > L $ and $\sum_{n=1}^\infty \E[\lVert w_n \rVert^2 |\mathcal{F}_n] < \infty$ by Assumption~\ref{asu3ii}, we can apply Lemma~\ref{lemma:Robbins-Siegmund-Chapter-3} to \eqref{inequality-prox-sequence} to conclude that $f(u_n)$ converges almost surely. The second statement follows immediately, since by Lemma~\ref{lemma:Robbins-Siegmund-Chapter-3},
\begin{equation}
\label{eq:finite-sum-difference-iterates}
\sum_{n=1}^\infty  \lVert \bar{u}_{n+1} - u_n \rVert^2 < \infty \quad \text{a.s.}, 
\end{equation}
which implies that for almost every sample path, $\lim_{n \rightarrow \infty} \lVert \bar{u}_{n+1} - u_n \rVert^2 = 0.$  

For the third statement, we have that there exists a subsequence $\{ u_{n_k}\}$ such that $u_{n_k} \rightarrow u$. We argue that then $\bar{u}_{n_k+1} \rightarrow u$. Since $\{ \bar{u}_{n_k+1}\}$ is bounded, there exists a weak limit point $\tilde{u}$ (potentially on a subsequence with the same labeling). Then, using weak lower semicontinuity of the norm as well as the rule $\langle a_n, b_n \rangle \rightarrow \langle a,b\rangle$ for $a_n \rightharpoonup a$ and $b_n \rightarrow b$,
\begin{align*}
 0 &= \lim_{k \rightarrow \infty} \lVert \bar{u}_{n_k+1} - u_{n_k}\rVert^2 = \lim_{k \rightarrow \infty} \lVert \bar{u}_{n_k+1} \rVert^2 - 2 \langle \bar{u}_{n_k+1}, u_{n_k} \rangle + \lVert u_{n_k} \rVert^2\\
 & = \liminf_{k \rightarrow \infty} \lVert \bar{u}_{n_k+1} \rVert^2 - 2 \langle \bar{u}_{n_k+1}, u_{n_k} \rangle + \lVert u_{n_k} \rVert^2 \\
 &\geq \lVert \tilde{u} \rVert^2 - 2 \langle \tilde{u}, u \rangle + \lVert u \rVert^2 = \lVert \tilde{u} - u \rVert^2 \geq 0,
\end{align*}
implying $u=\tilde{u}.$ It follows $\bar{u}_{n_k+1} \rightarrow u$ by assuming {$\lim_{k \rightarrow \infty} \lVert \bar{u}_{n_k+1} \rVert^2 \neq \lVert u \rVert^2$} and arriving at a contradiction. Now, by definition of the $\prox$ operator,
\begin{align*}
 \bar{u}_{n_k+1} &= \prox_{t h}(u_{n_k} - t \nabla j(u_{n_k}) )\\
 &= \argmin_{v \in H} \Big\lbrace h(v) + \frac{1}{2t} \lVert v- u_{n_k} + t \nabla j(u_{n_k}) \rVert^2 \Big\rbrace\\
 & = \argmin_{v \in H} \Big\lbrace h(v) + \langle \nabla j(u_{n_k}), v \rangle + \frac{1}{2t} \lVert v \rVert^2- \frac{1}{t} \langle v,u_{n_k}\rangle =:H(v)\Big\rbrace.
\end{align*}
Clearly, $\partial H(v)=\partial h(v) + \nabla j(u_{n_k}) + \tfrac{1}{t} (v-u_{n_k})$.
By optimality of $\bar{u}_{n_k+1}$ (see Fermat's rule, \cite[Theorem 16.2]{Bauschke2011}), $0 \in \partial H(\bar{u}_{n_k+1})$, or equivalently,
$$-\frac{1}{t} (\bar{u}_{n_k+1} - u_{n_k}) \in \nabla j(u_{n_k}) + \partial h(\bar{u}_{n_k+1}).$$
Taking the limit as $k \rightarrow \infty$, and using continuity of $\nabla j$, we conclude by strong-to-weak sequential closedness of $\textup{gra}(\partial h)$ that 
\begin{equation}
\label{eq:KKT-nonsmooth}
0 \in \nabla j(u) + \partial h(u),
\end{equation}
so therefore $u$ is a stationary point.
\end{proof}

\subsection{Stochastic Proximal Gradient Method - Decreasing Step Sizes}
\label{subsection:ODE-proof}
An obvious drawback of Algorithm~\ref{alg:PSG_Hilbert_Nonconvex} is the fact that step sizes are restricted to small steps bounded by a factor depending on the Lipschitz constant, which in applications might be difficult to determine. Additionally, the algorithm requires increasing batch sizes to dampen noise, which is unattractive from a complexity standpoint. In this section, we obtain convergence with a nonsmooth and convex term $h$ using the step size rule

\begin{equation}\label{eq:Robbins-Monro-stepsizes}
 t_n \geq 0, \quad \sum_{n=1}^\infty t_n = \infty, \quad \sum_{n=1}^\infty t_n^2 < \infty.
\end{equation}

This step size rule dampens noise enough so that increased sampling is not necessary.

We observe Problem \eqref{eq:ProblemFormulation-basic} with 
$$h(u) := \eta(u) + \delta_C(u).$$
For asymptotic arguments, it will be convenient to treat the term $\delta_C$ separately. To that end, we define
$$\varphi(u):=j(u)+\eta(u)$$
and note that $f(u) = \varphi(u) + \delta_C(u).$ The stochastic gradient $G(u,\xi):H\times \Xi \rightarrow H$ can be comprised of one or more samples as in the unconstrained case; see Remark~\ref{remark:general-assumptions}.   
The algorithm is now stated below.
\begin{algorithm}[H] 
\begin{algorithmic}[0] 
\STATE \textbf{Initialization:} $u_1 \in C$ 
\FOR{$n=1,2,\dots$}
\STATE Generate $\xi_{n} \in \Xi$, independent of $\xi_{1}, \dots, \xi_{n-1}$
\STATE Choose $t_n$ satisfying \eqref{eq:Robbins-Monro-stepsizes}
\STATE $u_{n+1}:=\prox_{t_n h}\left( u_n - t_n G(u_n,\xi_{n})\right)$
\ENDFOR
\end{algorithmic}
\captionof{algorithm}{Stochastic Proximal Gradient Method} 
\label{alg:PSG_Hilbert_Nonconvex_Decreasing_Steps}
\end{algorithm}
To prove convergence of Algorithm~\ref{alg:PSG_Hilbert_Nonconvex_Decreasing_Steps}, we will use the ODE method, which dates back to \cite{Kushner1978,Ljung1977}. While we use many ideas from \cite{Davis2018}, we emphasize that we generalize results to (possibly infinite-dimensional) Hilbert spaces and moreover, we handle the case when $j$ is the expectation. 

We define the set-valued map $S:C \rightrightarrows H$ by 
$$S(u) := - \nabla j(u) - \partial \eta(u) - N_C(u).$$
Additionally, we define the sequence of (single-valued) maps $S_n:C \rightarrow H$ for all $n$ by
$$S_n (u):= - \nabla j(u) - \frac{1}{t_n} \E[u - t_n G(u,\xi) - \prox_{t_n h}(u - t_n G(u,\xi))]. $$

In addition to Assumption~\ref{asu1}, the following assumptions will apply in this section.
\begin{assumption}
\label{assumptions:general-convergence-proof}
Let $\{ u_n\}$ and $\{G(u_n,\xi_n)\}$ be generated by Algorithm~\ref{alg:PSG_Hilbert_Nonconvex_Decreasing_Steps}. We assume\\ \setcounter{subassumption}{0}
\subasu \label{asu4i} The set $C$ is nonempty, bounded, convex, and closed.\\
\subasu \label{asu4ii} The function $\eta \in \Gamma_0(H)$ with $\textup{dom}(\eta) = H$ is locally Lipschitz and bounded below on $C$, and there exists a function $L_{\eta}: H \rightarrow \R$, which is bounded on bounded sets, satisfying 
\begin{equation}
\label{eq:local-Lipschitz-bound-h}
L_\eta(u) \geq \sup_{z:\eta(z) \leq \eta(u)} \frac{\eta(u)-\eta(z)}{\lVert u - z\rVert}.
\end{equation}
\subasu \label{asu4iii} There exists a function $M:H \rightarrow [0,\infty)$, which is bounded on bounded sets, such that
 $\E[\lVert G(u,\xi) \rVert^2] \leq M(u).$\\
\subasu \label{asu4iv} For any strongly convergent sequence $\{u_n\}$, $\E[ \sup_n \lVert G(u_n,\xi) \rVert] < \infty$ holds.\\
\subasu \label{asu4v} The set of critical values $\{ f(u): 0 \in \partial f(u)\}$ does not contain any segment of nonzero length.
\end{assumption}

\begin{remark}
To handle the infinite-dimensional case, we use assumptions that are generally more restrictive than in \cite{Davis2018}; we restrict ourselves to the case where $C$ and $\eta$ are convex and we assume higher regularity of $j$ in Assumption~\ref{asu1ii} to handle the case  $j(u) = \E[J(u,\xi)]$. However, we allow for bias $r_n$, which is not covered in \cite{Davis2018}. We note that $C$ does not need to be bounded if $\eta$ is Lipschitz continuous over $C$. Assumption~\ref{asu4ii} is satisfied if $\textup{dom}(\partial \eta) = H$ and $\partial \eta$ maps bounded sets to bounded sets; see also \cite[Proposition 16.17]{Bauschke2011} for equivalent conditions. The last assumption is technical but standard; see \cite[Assumption H4]{Ruszczynski1983}.
\end{remark}

The main result is the following, which we will prove in several parts. Throughout, we use the notation $g_n:=G(u_n,\xi_n)$.
\begin{theorem}
 \label{theorem:convergence-variance-reduced-stochastic-gradient-decreasing-steps}
Let Assumption~\ref{asu1} and Assumption~\ref{assumptions:general-convergence-proof} hold. Then
\begin{enumerate}
\item The sequence $\{ f(u_n) \}$ converges a.s.
\item Every strong accumulation point $u$ of the sequence $\{u_n\}$ is a stationary point with probability one, namely, $0 \in \partial f(u)$ a.s.
\end{enumerate}
\end{theorem}

\begin{lemma}
\label{lemma:recursion-relation}
The sequence $\{ u_n\}$ satisfies the recursion
\begin{equation}
\label{eq:fundamental-recursion}
 u_{n+1} = u_n + t_n (y_n - r_n +w_n),
\end{equation}
where $y_n = S_n(u_n)$ and $w_n=-\frac{1}{t_n}\E[\prox_{t_n h}(u_n - t_n g_n)|\mathcal{F}_n] + \frac{1}{t_n}\prox_{t_n h}(u_n - t_n g_n)$.
\end{lemma}

\begin{proof}
Note that $u_n$ and $r_n$ are $\mathcal{F}_n$-measurable, so $\E[g_n|\mathcal{F}_n] = \nabla j(u_n) + r_n$. Then 
 \begin{align*}
  &u_{n+1} - u_n =\prox_{t_n h}(u_n - t_n g_n) -  u_n\\
  & \quad= -t_n \E[g_n |\mathcal{F}_n] - \E[u_n - t_n g_n - \prox_{t_n h}(u_n - t_n g_n)|\mathcal{F}_n] \\
  &\quad \quad \quad - \E[\prox_{t_n h}(u_n - t_n g_n)|\mathcal{F}_n] + \prox_{t_n h}(u_n - t_n g_n)\\ 
  &\quad= t_n S_n(u_n) - t_n r_n - \E[\prox_{t_n h}(u_n - t_n g_n)|\mathcal{F}_n] + \prox_{t_n h}(u_n - t_n g_n),
 \end{align*}
where we used that $\xi_n$ is independent from $\xi_1, \dots, \xi_{n-1}$, so
\begin{equation}
\label{eq:measurability-yn}
\begin{aligned}
&\E[u_n - t_n g_n - \prox_{t_n h}(u_n - t_n g_n)|\mathcal{F}_n]\\
& \quad\quad = \E[u_n - t_n G(u_n,\xi) - \prox_{t_n h}(u_n - t_n G(u_n,\xi))].
\end{aligned}
\end{equation}
By definition of $y_n$ and $w_n$, we arrive at the conclusion.
\end{proof}

\begin{lemma}
\label{lemma:inequality-single-step}
For any $u \in C$, $g \in H$ and $t > 0$, we have for $\bar{u} = \prox_{t h}(u - t g)$ that
$$\frac{1}{t} \lVert \bar{u} - u \rVert \leq 2 L_\eta(u) + 2 \lVert g \rVert.$$
\end{lemma}

\begin{proof}
By definition of the proximity operator,
$$ \eta(\bar{u})  + \delta_C(\bar{u})+\frac{1}{2t} \lVert \bar{u} - (u-tg) \rVert^2 \leq \eta(u)+  \delta_C({u}) +\frac{1}{2t} \lVert u - (u-tg) \rVert^2,$$
or equivalently (note $\bar{u}, u \in C$),
$$ \eta(\bar{u}) + \frac{1}{2t} \lVert \bar{u} -u \rVert^2 + \langle \bar{u}-u, g\rangle \leq \eta(u). $$
By \eqref{eq:local-Lipschitz-bound-h}, in the case $\eta(u) \geq \eta(\bar{u})$, we obtain
\begin{equation}
 \label{lemma-inequality-single-step-proof1}
\frac{1}{t} \lVert \bar{u} - u \rVert^2 \leq 2 (\eta(u) -  \eta(\bar{u})) - 2\langle \bar{u} - u, g\rangle \leq 2L_\eta(u) \lVert \bar{u} - u \rVert + 2\lVert \bar{u} - u \rVert \lVert g \rVert.
\end{equation}
Notice that the last inequality \eqref{lemma-inequality-single-step-proof1} is trivial whenever $\eta(u) \leq \eta(\bar{u})$. This yields the conclusion.
\end{proof}

\begin{lemma}
\label{lemma:vanishing-diff-inclusion-approximation}
The sequence $\{ y_n\}$ is bounded a.s.
\end{lemma}

\begin{proof}
By the characterization of $y_n=S_n(u_n)$ from Lemma~\ref{lemma:recursion-relation}  and \eqref{eq:measurability-yn}, followed by Jensen's inequality, and the application of Lemma~\ref{lemma:inequality-single-step} in the fourth inequality, we get
\begin{equation} \label{eq:y_n_bounded}
\begin{aligned}
\lVert y_n \rVert &\leq \lVert \nabla j(u_n)\rVert + \lVert\tfrac{1}{t_n}\E[u_n - t_n g_n - \prox_{t_n h}(u_n - t_n g_n)|\mathcal{F}_n] \rVert\\
&\leq \lVert \nabla j(u_n)\rVert + \E \big[\lVert\tfrac{1}{t_n}\big(u_n - t_n g_n - \prox_{t_n h}(u_n - t_n g_n)\big) \rVert |\mathcal{F}_n\big]\\
&\leq \lVert \nabla j(u_n)\rVert + \E[\lVert g_n\rVert |\mathcal{F}_n]+ \E \big[ \lVert\tfrac{1}{t_n}\big(u_n - \prox_{t_n h}(u_n - t_n g_n)\big) \rVert |\mathcal{F}_n\big]\\
&\leq \lVert \nabla j(u_n)\rVert + \E[\lVert g_n\rVert |\mathcal{F}_n]+ 2 L_{\eta}(u_n) + 2 \E[\lVert g_n \rVert |\mathcal{F}_n]\\
&\leq \lVert \nabla j(u_n)\rVert + 3\sqrt{M(u_n)} + 2 L_{\eta}(u_n).
\end{aligned}
\end{equation}
The last step follows by $\E[\lVert g_n\rVert | \mathcal{F}_n] = \E[\lVert G(u_n,\xi)\rVert]$ and Assumption~\ref{asu4iii} with Jensen's inequality. We have from Assumption~\ref{asu1i} that $\{ u_n\}$ is bounded a.s.; therefore, all terms on the right-hand side of \eqref{eq:y_n_bounded} are bounded a.s. 
\end{proof}

For Lemma~\ref{lemma:vanishing-white-noise-terms}, we need the following result, which is a generalization of a convergence theorem for quadratic variations from~\cite[p.~111]{Williams1991} to Bochner spaces. The proof can be found in Sect.~\ref{subsection:auxiliary-proofs}.
\begin{lemma}
\label{lemma:quadratic-variations-bounded-imply-convergence}
Let $\{v_n\}$ be an $H$-valued martingale. Then $\{v_n\}$ is bounded in $L^2(\Omega,H)$ if and only if
\begin{equation}
\label{eq:quadratic-variations-proof}
\sum_{n=1}^\infty \E[\lVert v_{n+1}-v_n\rVert^2] < \infty,
\end{equation}
and when this is satisfied, 
$v_n \rightarrow v_\infty$ a.s.~as $n \rightarrow \infty$.
\end{lemma}

\begin{lemma}
\label{lemma:vanishing-white-noise-terms}
The series $\sum_{j=1}^N t_j w_j $ a.s.~converges to a limit as $N \rightarrow \infty$.
\end{lemma}

\begin{proof}
Recall the elementary inequality $\E[\lVert X - \E[X|\mathcal{F}_n]\rVert^2|\mathcal{F}_n]\leq \E[\lVert X\rVert^2|\mathcal{F}_n]$, which holds for any random variable $X$. By Lemma~\ref{lemma:recursion-relation} with $$X:=\tfrac{1}{t_n}(\prox_{t_n h}(u_n - t_n g_n)-u_n),$$ followed by Lemma~\ref{lemma:inequality-single-step} and Assumption~\ref{asu4iii}, we get
\begin{equation}
 \label{eq:bounds-second-moment-white-noise}
 \begin{aligned}
 \E[\lVert w_n \rVert^2 | \mathcal{F}_n]  &\leq \tfrac{1}{t_n^2}\E[\lVert \prox_{t_n h}(u_n - t_n g_n) - u_n\rVert^2 | \mathcal{F}_n] \\
 &\leq 4 (L_{\eta}(u_n))^2 +4M(u_n) <\infty.
 \end{aligned}
\end{equation}

Let $v_n := \sum_{j=1}^n t_j w_j$. We show that $v_n$ is a square integrable martingale, i.e.,~$v_n \in L^2(\Omega, H)$ for every $n$ and $\sup_{n} \E[\lVert v_n \rVert^2]<\infty.$ It is clearly a martingale, since for all $n$, $\E[w_n|\mathcal{F}_n] = 0$ and thus
$$\E[v_n|\mathcal{F}_n] = \E[t_n w_n |\mathcal{F}_n] + \sum_{j=1}^{n-1} t_j w_j = v_{n-1}.$$
To show that $v_n$ is square integrable, we use \eqref{eq:bounds-second-moment-white-noise} and the fact that $\E[v_n]=0$ for all $n$ to conclude that its quadratic variations are bounded. Indeed,
\begin{align*}
A_n &:= \sum_{j=2}^n \E[\lVert v_{j} -  v_{j-1} \rVert^2 | \mathcal{F}_{j}]= \sum_{j=2}^{n} t_j^2 \E[\lVert w_j \rVert^2|\mathcal{F}_j].
\end{align*}
Because of the condition \eqref{eq:Robbins-Monro-stepsizes}, we have that $\sup_n \E[A_n] < \infty.$ We have obtained that $\{v_n\}$ is square integrable, so by Lemma~\ref{lemma:quadratic-variations-bounded-imply-convergence}, 
it follows that $\{v_n\}$ converges a.s.~to a limit as $n\rightarrow \infty$.
\end{proof}

\begin{lemma}
 \label{lemma:easy-lemma}
 The following is true with probability one:
 \begin{equation}
\label{eq:Cauchy-sequence_u_n}
\lim_{n \rightarrow \infty} \lVert u_{n+1} -  u_n \rVert = 0.
\end{equation}
\end{lemma}

\begin{proof}
 This is a simple consequence of \eqref{eq:fundamental-recursion} and a.s.~boundedness of $ y_n$,  $r_n $, and $ w_n$ for all $n$ by Lemma~\ref{lemma:vanishing-diff-inclusion-approximation}, Assumption~\ref{asu1iii}, and Lemma~\ref{lemma:vanishing-white-noise-terms}, respectively.
\end{proof}

\begin{lemma}
\label{lemma:distance-between-sets-to-zero}
For any sequence $\{z_n\}$ in $C$ such that $z_n \rightarrow z$ as $n \rightarrow \infty$, it follows that
\begin{equation}
\label{eq:dist_statement}
\lim_{m\rightarrow \infty} d\left( \frac{1}{m} \sum_{n=1}^{m} S_{n}(z_{n}), S(z) \right) = 0 \quad \text{a.s.}
\end{equation}
\end{lemma}

\begin{proof}
Notice that $C$ is closed, so $z\in C$. The fact that $S(z)$ is nonempty, closed, and convex follows by these properties of $\nabla j(z)$, $\partial \eta(z)$, and $N_C(z)$.
We define $g_n^\xi:=G(z_n,\xi)$ and
\begin{equation}
\label{eq:S_n_tilde}
\tilde{S}_n(z_n, \xi):=- \nabla j(z_n) - \tfrac{1}{t_n} (z_n -t_n g_n^\xi - \prox_{t_n h}(z_n - t_n g_n^\xi)).
\end{equation}
Clearly, $\Exi[\tilde{S}_n(z_n, \xi)] = S_n(z_n).$
Now, by Jensen's inequality and convexity of the mapping $u \mapsto d (u,S(z))$,
\begin{align*}
d\left( \frac{1}{m} \sum_{n=1}^{m} S_{n}(z_{n}),S(z) \right) &\leq \frac{1}{m} \sum_{n=1}^{m} d(S_n(z_n),S(z))\\
&\leq \frac{1}{m} \sum_{n=1}^{m} \Exi \left[d(\tilde{S}_n(z_n, \xi),S(z))\right].
\end{align*}
Notice that $\bar{z} = \prox_{t h}(u)$ if and only if $0 \in \partial \eta (\bar{z})+ N_C(\bar{z}) +\tfrac{1}{t}(\bar{z}-u)$, so with
\begin{equation}
\label{eq:ch3-z-bar}
\bar{z}_n:=\prox_{t_n h}(z_n -t_n g_n^\xi),
\end{equation}
there exist 
$\zeta_{\eta,n} \in \partial \eta(\bar{z}_n)$ and $\zeta_{C,n} \in N_C(\bar{z}_n)$ such that
\begin{equation}
\label{eq:optimality-prox-step-in-dist-proof}
-(\zeta_{\eta,n} + \zeta_{C,n}) = \tfrac{1}{t_n}(\bar{z}_n - z_n + t_n g_n^\xi).
\end{equation}
Because $\{ z_n\}$ converges, it is contained in a bounded set. Hence, by Lemma~\ref{lemma:inequality-single-step}, we get
\begin{equation}
 \label{eq:boundedness_of_zeta_h_zeta_C}
\begin{aligned}
\lVert \zeta_{\eta,n} + \zeta_{C,n} \rVert &= \tfrac{1}{t_n} \lVert \bar{z}_n - z_n + t_n g_n^\xi \rVert \leq 2 L_{\eta}(z_n) + 3\lVert g_n^\xi\rVert,
\end{aligned}
\end{equation}
which must be almost surely finite by Assumption~\ref{asu4iv}. Now, by \eqref{eq:S_n_tilde} and \eqref{eq:ch3-z-bar}, followed by \eqref{eq:optimality-prox-step-in-dist-proof}, 
\begin{align*}
 d(\tilde{S}_n(z_n, \xi),S(z)) &= d(- \nabla j(z_n) + \tfrac{1}{t_n} (\bar{z}_n - z_n +t_n g_n^\xi),S(z))\\
 & =  d(- \nabla j(z_n) - \zeta_{\eta,n} - \zeta_{C,n}, S(z)).
\end{align*}
By the simple rule $d(u+v,A+B) \leq d(u,A)+d(v,B)$ for sets $A$ and $B$ and points $u, v\in H$, we get by definition of $S(z)$ that
$$ d(\tilde{S}_n(z_n, \xi),S(z)) \leq \lVert \nabla j(z_n)-\nabla j(z)\rVert+ d(\zeta_{\eta,n},\partial \eta(z)) + d(\zeta_{C,n},N_C(z)).$$
By strong-to-weak sequential closedness of $\textup{gra}(\partial \eta)$ and $\textup{gra}( N_C)$ as well as continuity of $\nabla j$, it follows that
\begin{equation}
\label{eq:distance-sequence-to-set-of-solutions-diff-incl.}
\lim_{n \rightarrow \infty} d(\tilde{S}_n(z_n, \xi),S(z)) = 0 \quad \text{a.s.}
\end{equation}
We show that $d(\tilde{S}_n(z_n, \xi),S(z))$ is almost surely bounded by an integrable function $\tilde{M}(z)$ for all $n$. Using elementary arguments and \eqref{eq:boundedness_of_zeta_h_zeta_C} in the third inequality,
 \begin{align*}
 & d(\tilde{S}_n(z_n, \xi),S(z))\\
 &\leq \quad d(-\nabla j(z_n) -\zeta_{\eta,n} - \zeta_{C,n},S(z))\\
 &\leq  \quad\lVert \nabla j(z_n) -\nabla j(z)\rVert + d (\zeta_{\eta,n} + \zeta_{C,n},\partial \eta(z) + N_C(z))\\
 & \leq  \quad\lVert \nabla j(z_n) - \nabla j(z)\rVert +  2 L_{\eta}(z_n) + 3\lVert g_n^\xi\rVert + d(0,\partial \eta(z) + N_C(z))\\
 &\leq  \quad\sup_{n \in \N} \left\lbrace\lVert \nabla j(z_n) - \nabla j(z)\rVert +  2 L_{\eta}(z_n) + 3\lVert g_n^\xi\rVert + d(0,\partial \eta(z) + N_C(z))\right\rbrace,
\end{align*}
which is almost surely bounded by Assumption~\ref{asu4ii} and Assumption~\ref{asu4iv}. By the dominated convergence theorem, it follows by \eqref{eq:distance-sequence-to-set-of-solutions-diff-incl.} that as $n \rightarrow\infty$, $\Exi [d(\tilde{S}_n(z_n, \xi),S(z))]\rightarrow 0$. Finally, \eqref{eq:dist_statement} follows from the fact that if $a_n \rightarrow 0$ as $n \rightarrow \infty$, it follows that $\tfrac{1}{m}\sum_{n=1}^m a_n \rightarrow 0$ as $m \rightarrow \infty$.
\end{proof}

Now we will show a compactness result, adapted from \cite{Duchi2018}, namely that in the limit, the time shifts of the linear interpolation of the sequence $\{ u_n\}$ can be made arbitrarily close to trajectories, or solutions, of the differential inclusion
\begin{equation}
\label{eq:differential-inclusion}
\dot{z}(t) \in S(z(t)).
\end{equation}
The set $C(I,H)$ denotes the space of continuous functions from $I$ to $H$. We recall that if $z(\cdot) \in C([0,\infty),H)$ satisfies \eqref{eq:differential-inclusion} and is absolutely continuous on any compact interval $[a,b] \subset (0,\infty)$, it is called a strong solution. The existence and uniqueness of this solution is guaranteed by the following result.

\begin{proposition}
\label{thm:Brezis-well-posedness-differential-inclusion} 
For every $z_0=z(0)\in C$ there exists a unique strong solution $z \in C([0,\infty),H)$ to the differential inclusion \eqref{eq:differential-inclusion}.
\end{proposition}

\begin{proof}
The function $u \mapsto  \eta(u) + \delta_C(u)$ is proper, convex, and lower semicontinuous and $B := - \nabla j$ is Lipschitz continuous. Therefore, by \cite[Proposition 3.12]{Brezis1973}, the statement follows. 
\end{proof}

For the next result, we set $s_n: = \sum_{j=1}^{n-1} t_j$ and define the linear interpolation $u:[0,\infty) \rightarrow H$ of iterates as well as the piecewise constant extension $y:[0,\infty) \rightarrow H$ of the sequence $\{y_n\}$ via
\begin{equation}
\label{eq:interpolation-sequences-un-yn}
u(t) := u_n + \frac{t-s_n}{s_{n+1}-s_n} (u_{n+1} - u_n), \quad y(t) := y_n, \quad \forall t \in [s_n,s_{n+1}), \forall n \in \N.
\end{equation}
The time shifts of $u(\cdot)$ are denoted by $u(\cdot+\tau)$ for $\tau>0$. We define  $u^\tau:[0,\infty) \rightarrow H$ by 
\begin{equation}
\label{eq:absolutely-continuous-trajectory}
u^\tau(t):=u(\tau) + \int_\tau^{t} y(s) \D s
\end{equation}
as the solution to the ODE
$$\dot{u}^\tau(\cdot) = y(\cdot), \quad u^\tau(\tau) = u(\tau),$$
which is guaranteed to exist by \cite[Theorem 1.4.35]{Cazenave1998}. 

\begin{theorem}
\label{theorem:compactness-result}
For any $T>0$ and any nonnegative sequence $\{ \tau_n\}$, the sequence of the time shifts $\{ u(\cdot +\tau_n)\}$ is relatively compact in $C([0,T],H)$. If $\tau_n \rightarrow \infty$, all limit points $\bar{u}(\cdot)$ of the time shifts $\{ u(\cdot+\tau_n)\}$ are in $C([0,T],H)$ and there exists a $\bar{y}:[0,T] \rightarrow H$ such that $\bar{y}(t) \in S(\bar{u}(t))$ and $\bar{u}(t) = \bar{u}(0) + \int_0^t \bar{y}(s) \D s.$
\end{theorem}

\begin{proof}
\textbf{Relative compactness of time shifts.} We first claim that for all $T>0$,
\begin{equation}
\label{eq:compactness-result-claim1}
 \lim_{\tau \rightarrow \infty} \sup_{t \in [\tau, \tau+T]} \lVert u^{\tau}(t)-u(t) \rVert = 0 \quad \text{a.s.}
\end{equation}
We consider a fixed (but arbitrary) sample path $\omega = (\omega_1, \omega_2, \dots)$ throughout the proof.
Let $p:=\min\{n:s_n \geq \tau\}$ and $q:=\max\{n:s_n \leq t\}$. By \eqref{eq:absolutely-continuous-trajectory} and \eqref{eq:interpolation-sequences-un-yn},
\begin{equation}
\label{eq:proof-compactness-absolutely_continuous}
\begin{aligned}
 u^\tau(t) &= u(\tau) + \int_{\tau}^t y(s) \D s = u(\tau) + \int_{\tau}^{s_{p}} y(s) \D s + \sum_{\ell=p}^{q-1} t_\ell y_\ell + \int_{s_{q}}^t y(s) \D s.
\end{aligned}
\end{equation}
Notice that due to the recursion~\eqref{eq:fundamental-recursion}, 
\begin{equation}
\label{eq:proof-compactness-2}
\sum_{\ell = p}^{q-1} t_\ell y_\ell = u_{q} - u_{p} - \sum_{\ell=p}^{q-1} t_\ell (w_\ell-r_\ell).
\end{equation}
Plugging \eqref{eq:proof-compactness-2} into \eqref{eq:proof-compactness-absolutely_continuous}, we get
\begin{align*}
 u^\tau(t) - u(t) &= u(\tau) + u_{q} - u_{p} - u(t) + \int_{\tau}^{s_{p}} y(s) \D s \\
 &\quad \quad - \sum_{\ell=p}^{q-1} t_\ell (w_\ell - r_\ell)+ \int_{s_{q}}^t y(s) \D s.
\end{align*}
Therefore,
\begin{align*}
\lVert u^\tau(t) - u(t) \rVert &\leq \left\lVert u(\tau) - u_{p} +\int_{\tau}^{s_{p}}y(s) \D s \right\rVert + \left\lVert u_{q} - u(t) +\int_{s_{q}}^{t} y(s) \D s\right\rVert \\
&\quad \quad + \left\lVert \sum_{\ell=p}^{q-1} t_\ell w_\ell\right\rVert + \left\lVert \sum_{\ell=p}^{q-1} t_\ell r_\ell\right\rVert.
\end{align*}
Note that by \eqref{eq:interpolation-sequences-un-yn}, it follows that
\begin{align*}
\lVert u(\tau) - u_{p} \rVert &\leq \lVert u_{p-1} - u_{p}\rVert = t_{p-1} \lVert y_{p-1} - r_{p-1} + w_{p-1}\rVert, \\
\lVert u_{q} - u(t)\rVert &\leq \lVert u_{q} - u_{q+1}\rVert = t_{q}\lVert y_{q} - r_{q} + w_{q}\rVert.
\end{align*}
Moreover, by \eqref{eq:interpolation-sequences-un-yn}, we have 
$$\left\lVert \int_{\tau}^{s_{p}} y(s) \D s\right\rVert \leq t_{p-1} \lVert y_{p-1}\rVert \quad \text{and} \quad \left\lVert \int_{s_{q}}^{t} y(s) \D s\right\rVert \leq t_{q} \lVert y_{q}\rVert.$$
Therefore, 
\begin{equation}
\label{eq:inequality-ODE-solutions-interpolation}
\begin{aligned}
 \lVert u^\tau(t) - u(t) \rVert &\leq  t_{p-1} (2\lVert y_{p-1}\rVert +\lVert r_{p-1}\rVert+ \lVert w_{p-1}\rVert)\\
 & \qquad +t_{q} (2\lVert y_{q}\rVert +\lVert r_{q}\rVert+ \lVert w_{q}\rVert) + \left\lVert \sum_{\ell=p}^{q-1} t_\ell w_\ell\right\rVert + \left\lVert \sum_{\ell=p}^{q-1} t_\ell r_\ell\right\rVert.
\end{aligned}
\end{equation}
We take the limit $p,q \rightarrow \infty$ on the right-hand side of \eqref{eq:inequality-ODE-solutions-interpolation} and observe that by
Lemma~\ref{lemma:vanishing-diff-inclusion-approximation}, $\lim_{n \rightarrow \infty} \sup_{m \geq n} t_m \lVert y_m \rVert = 0$ and by
Lemma~\ref{lemma:vanishing-white-noise-terms}, we have $\lim_{n \rightarrow \infty} \sup_{m \geq n} \lVert  \sum_{\ell=n}^{m-1} t_\ell w_\ell\rVert = 0$ as well as $\lim_{n \rightarrow \infty} \sup_{m\geq n} t_m \lVert w_m\rVert$. By Assumption~\ref{asu1iii}, we have $\lim_{n \rightarrow \infty} \sup_{m \geq n} \left\lVert  \sum_{\ell=n}^{m-1} t_\ell r_\ell \right\rVert = 0.$ We have shown \eqref{eq:compactness-result-claim1}, so it follows that the set 
$$A:=\{u^\tau(\cdot): \tau \in [0,\infty)\}$$ 
is a family of equicontinuous functions. 

To invoke the Arzel\`{a}--Ascoli theorem, we first show that the set 
$$A(t):=\{ u^{\tau}(t): \tau \in [0,\infty)\}$$
is relatively compact for all $t\in [0,T]$, $T>0$. We show this by proving that arbitrary sequences in $A(t)$ have a Cauchy subsequence, which converge in $H$ by completeness of $H$. To this end, let $\varepsilon>0$ be arbitrary and observe first the case $\tau_n \rightarrow \infty.$ Let $n_k$ be the index such that $\tau_k \in [s_{n_k}, s_{n_k+1})$ and
$$u^{\tau_k}(t) = u_{n_k} + \frac{\tau_k - s_{n_k}}{s_{n_k+1}-s_{n_k}} (u_{n_k+1} - u_{n_k}) + \int_{\tau_k}^t y(s) \D s.$$
Similarly, let $m_j$ be the index such that $\tau_j \in [s_{m_j},s_{m_j+1})$. Thus we have
\begin{equation}
\label{eq:first-cauchy-inequality}
\begin{aligned}
 &\lVert u^{\tau_k}(t) - u^{\tau_j}(t)\rVert \\
 &\qquad \leq \left\lVert \frac{\tau_k - s_{n_k}}{s_{n_k+1}-s_{n_k}} ( u_{n_k+1} - u_{n_k}) -   \frac{\tau_j - s_{m_j}}{s_{m_j+1}-s_{m_j}} ( u_{m_j+1} - u_{m_j}) \right\rVert \\
 &\qquad \qquad  +  \left\lVert u_{n_k} - u_{m_j} + \int_{\tau_k}^{\tau_j} y(s) \D s\right\rVert.
\end{aligned}
\end{equation}
Using \eqref{eq:proof-compactness-2}, we get (w.l.o.g.~$\tau_k \leq \tau_j$)
\begin{equation}
 \label{eq:second-cauchy-inequality}
\begin{aligned}
 \left\lVert   u_{n_k} - u_{m_j} + \int_{\tau_k}^{\tau_j} y(s) \D s\right\rVert &\leq
 \lVert   u_{n_k} - u_{n_k+1} \rVert + \left\lVert  \int_{\tau_k}^{s_{n_k+1}} y(s) \D s \right\rVert\\
 & \quad\quad  +\left\lVert \int_{s_{m_j}}^{\tau_j} y(s) \D s \right\rVert +\left\lVert \sum_{\ell=n_k+1}^{m_j-1} t_\ell (w_\ell - r_\ell)\right\rVert.
\end{aligned}
\end{equation}
Combining \eqref{eq:first-cauchy-inequality} and \eqref{eq:second-cauchy-inequality}, and observing that $\left|\tfrac{\tau_k - s_{n_k}}{s_{n_k+1}-s_{n_k}}\right| \leq 1$ as well as $\left|\tfrac{\tau_j - s_{m_j}}{s_{m_j+1}-s_{m_j}}\right| \leq 1$, we obtain
\begin{equation}
\begin{aligned}
  \label{eq:third-cauchy-inequality}
 \lVert u^{\tau_k}(t) - u^{\tau_j}(t)\rVert  &\leq 2 \lVert u_{n_k+1} - u_{n_k}\rVert + \lVert u_{m_j+1} - u_{m_j}\rVert   + t_{n_k} \lVert y_{n_k}\rVert  \\
 & \quad\quad  + t_{m_j} \lVert y_{m_j}\rVert + \left \lVert \sum_{\ell=n_k+1}^{m_j-1} t_\ell (w_\ell - r_\ell)\right\rVert.
 \end{aligned}
 \end{equation}
By Lemma~\ref{lemma:easy-lemma} as well as convergence of the other terms on the right-hand side of \eqref{eq:third-cauchy-inequality}, for $\varepsilon >0$ there exists a $N$ such that for all $k, j > N$, 
$\lVert u^{\tau_k}(t) - u^{\tau_j}(t)\rVert \leq \varepsilon$ for all $k, j > N$ and thus $\{ u^{\tau_n}(t)\}$ has a Cauchy subsequence for $\tau_n \rightarrow \infty$. Now we observe the case where the sequence $\{\tau_n \}$ is bounded. Then $\tau_n \rightarrow \bar{\tau}$ for some $\bar{\tau}>0$ at least on a subsequence (with the same labeling). By convergence of $\{\tau_n\}$ we get that $m_j=n_k$ for $k, j \geq N$ and $N$ large enough. Therefore  \eqref{eq:first-cauchy-inequality} reduces to 
\begin{equation}
\label{eq:fifth-cauchy-inequality}
\lVert u^{\tau_k}(t) - u^{\tau_j}(t)\rVert \leq \left\lvert  \frac{\tau_k - \tau_j}{s_{n_k+1}-s_{n_k}}\right\rvert \lVert u_{n_k+1}-u_{n_k} \rVert + \left\lVert \int_{\tau_k}^{\tau_j} y(s) \D s \right\rVert. 
\end{equation}
We can bound terms on the right-hand side of \eqref{eq:fifth-cauchy-inequality} as before to obtain that $\{ u^{\tau_n}(t) \}$ has a Cauchy subsequence. We have shown that $A(t)$ is relatively compact for all $t\in [0,T]$, $T>0$, so by the Arzel\`{a}--Ascoli theorem, it follows that the set $A$ is relatively compact.

Now, the relative compactness of the set of time shifts $\{ u(\cdot +\tau):  \tau \in [0,\infty)\}$ follows from the relative compactness of the set $A$. Indeed, for any sequence $\{u^{\tau_n}(\cdot + \tau_n)\}$ there exists a convergent subsequence such that $u^{\tau_{n_k}}(\cdot + \tau_{n_k}) \rightarrow \bar{u}(\cdot)$ for some $\bar{u}(\cdot) \in C([0,T],H)$. Now, for the time shift $u(\cdot +\tau_{n_k})$, we have 
\begin{align*}
&\sup_{t \in [0,T]}\lVert u(t+\tau_{n_k}) - \bar{u}(t)\rVert \\
&\qquad \leq  \sup_{t \in [0,T]}\lVert u(t+\tau_{n_k}) -  u^{\tau_{n_k}}(t+\tau_{n_k})\rVert + \sup_{t \in [0,T]}\lVert u^{\tau_{n_k}}(t+\tau_{n_k}) - \bar{u}(t)\rVert,\end{align*}
so it follows that $u(\cdot +\tau_{n_k}) \rightarrow \bar{u}(\cdot)$ in $C([0,T],H)$ as $\tau_{n_k}\rightarrow \infty$ by convergence of $u^{\tau_{n_k}}(\cdot)$ and \eqref{eq:compactness-result-claim1}. If $\tau_{n_k} \rightarrow \bar{\tau}$, then $u(\cdot+\tau_{n_k}) \rightarrow u(\cdot+\bar{\tau})$ by uniform continuity of $u(\cdot)$ on $[0,\bar{\tau}+T].$

\textbf{Limit points are trajectories of the differential inclusion.}
Let $\{ \tau_n\}$ be a sequence such that as $\tau_n \rightarrow \infty$, $u^{\tau_n}(\cdot+\tau_n) \rightarrow \bar{u}(\cdot)$ in $C([0,T],H)$ (potentially on a subsequence). The sequence $\{ y(\cdot+\tau_n)  \} \subset L^2([0,T], H)$ is bounded by boundedness of $\{ y_n\}$, and since $L^2([0,T], H)$ is a Hilbert space, there exists a subsequence $\{ n_k\}$ such that $y(\cdot+\tau_{n_k}) \rightharpoonup \bar{y}(\cdot)$ in $L^2([0,T],H)$ for some $\bar{y} \in L^2([0,T],H)$. Notice that for $\{\tau_{n_k}\}$, by \eqref{eq:absolutely-continuous-trajectory} it follows that
\begin{equation}
\label{eq:time-shifted-subsequence}
u^{\tau_{n_k}}(t+\tau_{n_k}) = u^{\tau_{n_k}}(\tau_{n_k}) + \int_0^t y(s+\tau_{n_k}) \D s.
\end{equation}
By \eqref{eq:compactness-result-claim1}, $u^{\tau_{n_k}}( \cdot+\tau_{n_k}) \rightarrow \bar{u}(\cdot)$ in $C([0,T],H)$ as $k \rightarrow \infty$. Taking $k \rightarrow \infty$ on both sides of \eqref{eq:time-shifted-subsequence} we get, due to $y(\cdot+\tau_{n_k}) \rightharpoonup \bar{y}(\cdot)$ for $ t \in [0,T]$, that
$$\bar{u}(t) = \bar{u}(0) + \int_0^t \bar{y}(s) \D s.$$
Now, we will show that $\bar{y}(t) \in S(\bar{u}(t))$ for a.e.~$t \in [0,T]$. By the Banach-Saks theorem (cf.~\cite{Okada1984}), there exists a subsequence of $\{ y(\cdot+\tau_{n_k})\}$ (where we use the same notation for the sequence as its subsequence) such that
\begin{equation}
\label{eq:Banach-Saks-sum}
\lim_{m \rightarrow \infty}\frac{1}{m} \sum_{k=1}^m y(\cdot+\tau_{n_k}) = \bar{y}(\cdot).
\end{equation}
Recall that $y_n = S_n(u_n)$ by Lemma~\ref{lemma:recursion-relation} and set $\ell_k^t:= \max\{ \ell: s_\ell \leq t+\tau_{n_k}\}.$ Then we have
$$y(t+\tau_{n_k}) = y(s_{\ell_k^t}) = y_{\ell_k^t} = S_{\ell_k^t}(u_{\ell_k^t}).$$
Therefore, since $t+\tau_{n_k} \in [\ell_k^t, \ell_k^t+1]$,
\begin{equation}
\label{eq:convergence-of-iterates-to-limit}
\begin{aligned}
 \lVert u(s_{\ell_k^t}) - \bar{u}(t)\rVert &\leq \lVert u(s_{\ell_k^t}) - u(t+\tau_{n_k}) \rVert + \lVert u(t+\tau_{n_k}) - \bar{u}(t) \rVert\\
 & \leq \lVert u(s_{\ell_k^t}) - u(s_{\ell_k^t+1}) \rVert + \lVert  u(t+\tau_{n_k}) - \bar{u}(t) \rVert\\
 & \leq t_{\ell_k^t} (\lVert y_{\ell_k^t}\rVert + \lVert r_{\ell_k^t}\rVert+ \lVert w_{\ell_k^t}\rVert) +\lVert  u(t+\tau_{n_k}) - \bar{u}(t) \rVert,
\end{aligned}
\end{equation}
which a.s.~converges to zero as $k \rightarrow \infty$, since $u(\cdot+\tau_{n_k}) \rightarrow \bar{u}(\cdot)$ and the fact that $t_{n} \rightarrow 0$ by \eqref{eq:Robbins-Monro-stepsizes} (combined  a.s.~boundedness of $y_n, r_n$, and $w_n$ for all $n$ by Lemma~\ref{lemma:vanishing-diff-inclusion-approximation}, Assumption~\ref{asu1iii}, and Lemma~\ref{lemma:vanishing-white-noise-terms}, respectively). Now, using $y(t+\tau_{n_k}) = y_{\ell_k^t}$, we get
\begin{align*}
 &d(\bar{y}(t),S(\bar{u}(t)))\\
 &\qquad \leq \left\lVert  \frac{1}{m} \sum_{k=1}^m y(t+\tau_{n_k}) - \bar{y}(t) \right\rVert +  d\left(\frac{1}{m} \sum_{k=1}^m y(t+\tau_{n_k}), S(\bar{u}(t))\right) \\
&\qquad \leq  \left\lVert  \frac{1}{m} \sum_{k=1}^m y(t+\tau_{n_k}) - \bar{y}(t) \right\rVert +  d\left(\frac{1}{m} \sum_{k=1}^m S_{\ell_k^t}(u(s_{\ell_k^t})), S(\bar{u}(t))\right),
\end{align*}
which converges to zero as $m \rightarrow \infty$ by \eqref{eq:Banach-Saks-sum} and Lemma~\ref{lemma:distance-between-sets-to-zero}, where we note that {$u(s_{\ell_k^t}) \rightarrow \bar{u}(t)$} as $k \rightarrow \infty$ by \eqref{eq:convergence-of-iterates-to-limit}. Since $S(\bar{u}(t))$ is a closed set and the sample path was chosen to be arbitrary, we have that the statement must be true with probability one. 
\end{proof}

Now, we show that there is always a strict decrease in $\varphi$ along a trajectory that originates at a noncritical point $z(0)$.
\begin{lemma}
\label{lemma:descent-property}
Whenever $z: [0,\infty) \rightarrow C$ is a trajectory satisfying the differential inclusion \eqref{eq:differential-inclusion}
and $0 \not\in S(z(0))$, then there exists a $T>0$ such that 
\begin{equation}
\label{eq:descent-property}
\varphi(z(T)) < \sup_{t \in [0,T]} \varphi(z(t)) \leq \varphi(z(0)).
\end{equation}
\end{lemma}
\begin{proof}We modify the proof from \cite[Lemma 5.2]{Davis2018}. Let $\delta, \tau$ satisfying  $0<\delta<\tau$ be fixed but arbitrary. From Theorem~\ref{thm:Brezis-well-posedness-differential-inclusion} we have that $z$ is absolutely continuous on $[\delta,\tau]$. It is straightforward to show that $\varphi\circ z: [\delta,\tau] \rightarrow \R$ is absolutely continuous, since $C$ is bounded and $\varphi$ is a composition of a locally Lipschitz map with an absolutely continuous function. Therefore, by Rademacher's theorem, it is  differentiable for almost every $t\in [\delta,\tau].$ On the other hand, notice that since $\eta$ is locally Lipschitz near $z(t)$ and convex, it is Clarke regular, 
so the chain rule $\partial(\eta \circ z)(t) = \partial \eta(z(t)) \circ \dot{z}(t)$ holds by \cite[Theorem 2.3.10]{Clarke1990}. 
The chain rule for $j$ holds by differentiability. 
Therefore for almost every $t$, it follows for all $v \in \partial \varphi(z(t))$ that
\begin{equation}
 \label{eq:chain-rule-step}
 (\varphi \circ z)'(t) = \partial (\varphi \circ z)(t)  
 = (\nabla j(z(t)) + \partial \eta(z(t)))\circ \dot{z}(t) = \langle v, \dot{z}(t) \rangle.
\end{equation}
We now observe the following property for the subdifferential of $\delta_C$, namely, 
\begin{equation}\label{eq:chain-rule-step2}
\langle v, \dot{z}(t)\rangle =0  \quad \forall v\in N_C(z(t)). 
\end{equation}
Indeed, since $z(\cdot)$ takes values in $C$ and by definition of the subdifferential, for all $r\geq 0$ it follows that
\begin{equation*}
0=\delta_C(z(t+r))-\delta_C(z(t)) \geq \langle v, z(t+r)-z(t)\rangle.
\end{equation*}
Hence,
\begin{equation*}
0\geq \lim_{r\rightarrow 0^+} \left\langle v,\frac{z(t+r)-z(t)}{r}\right\rangle =  \langle v,\dot{z}(t)\rangle.
\end{equation*}
The reverse inequality can be obtained by using the left limit of the difference quotient, and we get \eqref{eq:chain-rule-step2}.
By \eqref{eq:chain-rule-step} and \eqref{eq:chain-rule-step2}, we obtain for a.e.~$t$ that
\begin{equation}\label{eq:chain_rule}
\langle v, \dot{z}(t)\rangle = \partial (\varphi \circ z)(t) \quad \forall v\in -S(z(t)).
\end{equation}
We now show that $\lVert \dot{z}(t) \rVert = d(0,S(z(t)))$. Trivially, $d(0,S(z(t))) \leq \lVert \zeta -0\rVert$ for all {$\zeta \in S(z(t))$,} so it follows that $d(0,S(z(t))) \leq \lVert \dot{z}(t)\rVert.$ Notice that for all $v, w \in \partial \varphi(z(t))$, by \eqref{eq:chain-rule-step}, $0 = \langle v-w, \dot{z}(t)\rangle.$ Setting {$W := \text{span}(\partial \varphi(z(t)) - \partial \varphi(z(t)))$,} we get $\dot{z}(t) \in W^\perp$. Clearly, {$-\dot{z}(t) \in (-\dot{z}(t) + W) \cap W^\perp$} so $\lVert \dot{z}(t)\rVert \leq d(0, -\dot{z}(t)+W)$. Since $\partial \varphi(z(t)) \subset \dot{z}(t)+W$, it follows {$\lVert \dot{z}(t) \rVert \leq d(0, \partial \varphi(z(t)))$} and we get $\lVert \dot{z}(t) \rVert = d(0,S(z(t)))$.

Now, notice that by \eqref{eq:chain_rule} and the fact that $\dot{z}(t) \in S(z(t))$, we have for a.e.~$t$ that
$$ \partial({\varphi} \circ z)(t) = -\lVert \dot{z}(t) \rVert^2 = -d(0,S(z(t)))^2.$$
Since $\varphi \circ z$ is absolutely continuous on $[\delta,\tau]$,
\begin{equation}
\label{eq:distance-in-proof-for-chain-rule}
\varphi( z (\tau)) = \varphi( z (\delta)) - \int_{\delta}^\tau d(0, S(z(s)))^2 \D s
\end{equation}
and hence $\varphi(z(\delta)) \geq \varphi(z(\tau))$. Using the continuity of $\varphi \circ z$, and the fact that $0<\delta<\tau$ were arbitrarily chosen, we get $\varphi(z(0)) \geq \varphi(z(t))$ for all $t>0$. To finish the proof, we must find some $T>0$ such that $\varphi(z(T)) < \sup_{t \in [0,T]}\varphi(z(t))$. Suppose that $d(0,S(z(t))) = 0$ for a.e.~$t \in [0,T]$ for all $T>0$. Since $\lVert \dot{z}(t) \rVert = d(0,S(z(t)))$ then $z \equiv z(0)$. This is a contradiction, since $\dot{z}(\cdot) \in S(z(\cdot))$ and $0 \not\in S(z(0))$. By \eqref{eq:distance-in-proof-for-chain-rule}, we conclude that there exists a $T>0$ such that \eqref{eq:descent-property} holds.
\end{proof}

The following proof is standard, but we need to make several arguments differently in the infinite-dimensional setting. We will proceed as in \cite{Davis2018}. We define the level sets of $\varphi$ as 
$$\mathcal{L}_r := \{u\in H: \varphi(u) \leq r \}.$$

\begin{proposition}
 For all $\varepsilon>0$ there exists a $N$ such that for all $n\geq N$, if $u_n \in \mathcal{L}_\varepsilon$, then $u_{n+1} \in \mathcal{L}_{2 \varepsilon}$ a.s.
\end{proposition}

\begin{proof}
First, we remark that $\varphi$ is uniformly continuous on $V$, since $\eta(\cdot)$ satisfies \eqref{eq:local-Lipschitz-bound-h} and, in turn, is Lipschitz continuous on $V$, as well as the fact that $j$ is Lipschitz continuous on $V$. Therefore, for any $\varepsilon>0$ there exists a $\delta>0$ such that if 
$\lVert u_{n+1}- u_n \rVert < \delta$, then $ |\varphi(u_{n+1}) - \varphi(u_n)| < \varepsilon.$
Now, we choose $N$ such that $\lVert u_{n+1} -  u_n \rVert < \delta$ for all $n\geq N$, which is possible by Lemma~\ref{lemma:easy-lemma}. Then it must follow that $|\varphi(u_{n+1}) - \varphi(u_n)| < \varepsilon$ for all $n\geq N$ as well. Now, since $u_n \in \mathcal{L}_\varepsilon$, it follows that $\varphi(u_{n+1}) \leq 2\varepsilon$, so therefore $u_{n+1} \in \mathcal{L}_{2\varepsilon}$.
\end{proof}

\begin{lemma}
\label{lemma:same-limits}
The following equalities hold.
\begin{equation}
 \label{eq:limit-sequence-limit-trajectory}
 \liminf_{n \rightarrow \infty} \varphi(u_n) = \liminf_{t\rightarrow \infty} \varphi(u(t)) \quad \text{and} \quad  \limsup_{n \rightarrow \infty} \varphi(u_n) = \limsup_{t\rightarrow \infty} \varphi(u(t)).
\end{equation}
\end{lemma}

\begin{proof}
We argue that $\liminf_{n \rightarrow \infty} \varphi(u_n) \leq \liminf_{t\rightarrow \infty} \varphi(u(t))$; the other direction is clear by construction of $u(\cdot)$ from \eqref{eq:interpolation-sequences-un-yn}. Let $\{\tau_n\}$ be a sequence such that $\tau_n \rightarrow \infty$, $\lim_{n \rightarrow \infty} u(\tau_n) = \bar{u}$ for some $\bar{u} \in H$, and $\liminf_{n \rightarrow \infty} \varphi(u(\tau_n)) = \varphi(\bar{u})$.  With $k_n := \max \{ n: t_k \leq \tau_n\}$, we get
\begin{align*}
 \lVert u_{k_n} - \bar{u} \rVert \leq \lVert u_{k_n} - u(\tau_n) \rVert + \lVert u(\tau_n) - \bar{u} \rVert \leq  \lVert u_{k_n} - u_{k_{n+1}} \rVert + \lVert u(\tau_n) - \bar{u} \rVert,
\end{align*}
which converges to zero as $n\rightarrow \infty$ by \eqref{eq:Cauchy-sequence_u_n} and convergence of the sequence $\{ u(\tau_n)\}.$ Therefore $u_{k_n} \rightarrow \bar{u}$ and so by continuity of $\varphi$, it follows that
$$\liminf_{t \rightarrow \infty} \varphi(u(t)) = \varphi(\bar{u}) = \lim_{n \rightarrow \infty} \varphi(u_{k_n}) \geq \liminf_{n \rightarrow \infty} \varphi(u_n).$$
Analogous arguments can be made for the claim $$\limsup_{n \rightarrow \infty} \varphi(u_n) = \limsup_{t\rightarrow \infty} \varphi(u(t)).$$
\end{proof}

\begin{lemma}
\label{lemma:exit-lemma} 
Only finitely many iterates $\{ u_n\}$ are contained in $H \backslash \mathcal{L}_{2\varepsilon}.$
\end{lemma}

\begin{proof}
We choose $\varepsilon>0$ such that $\varepsilon \notin \varphi(S^{-1}(0)),$ which is possible for arbitrarily small $\varepsilon$ by Assumption~\ref{asu4v}, where we note that $\varphi(S^{-1}(0)) = f(S^{-1}(0))$. We construct the process given by the recursion
\begin{align*}
 i_1 & := \min \{n : u_n \in \mathcal{L}_\varepsilon \text{ and } u_{n+1}\in \mathcal{L}_{2\varepsilon} \backslash \mathcal{L}_{\varepsilon}\},\\
 e_1 & := \min \{ n: n > i_1 \text{ and } u_n \in H\backslash \mathcal{L}_{2\varepsilon}\}, \\
 i_2 & := \min \{n: n > e_1 \text{ and } u_n \in \mathcal{L}_\varepsilon \},
\end{align*}
and so on. We argue by contradiction and recall that $s_n = \sum_{j=1}^{n-1} t_j$. Suppose infinitely many $\{ u_n\}$ are in $H \backslash \mathcal{L}_{2\varepsilon}$, then it must follow that $i_j \rightarrow \infty$ as $j \rightarrow \infty$. By 
Theorem~\ref{theorem:compactness-result}, $\{u(\cdot+s_{i_j})\}$ is relatively compact in $C([0,T], H)$ for all $T>0$ and there exists a subsequence (with the same labeling) and limit point $z(\cdot)$ such that $z(\cdot)$ is a trajectory of \eqref{eq:differential-inclusion}. 
Now, since by construction $\varphi(u_{i_j}) \leq \varepsilon$ and $\varphi(u_{i_j+1}) > \varepsilon$, it follows that
\begin{equation}
\label{eq:simple-inequality-proof-exit-lemma}
\begin{aligned}
 \varepsilon \geq \varphi(u_{i_j}) &= \varphi(u_{i_j+1}) + \varphi(u_{i_j}) - \varphi(u_{i_j+1})\\
 &\geq \varepsilon + \varphi(u_{i_j})- \varphi(u_{i_j+1}).
\end{aligned}
\end{equation}
Recall that $\lim_{j \rightarrow \infty} u_{i_j} = u(\cdot+s_{i_j}) = z(0)$. Taking the limit $j \rightarrow \infty$ on both sides of \eqref{eq:simple-inequality-proof-exit-lemma}, by continuity of $\varphi$, we get 
$$\lim_{j \rightarrow \infty} \varphi(u_{i_j}) = \varphi(z(0)) = \varepsilon,$$
meaning $z(0)$ is not a critical point of $\varphi$. Thus we can invoke Lemma~\ref{lemma:descent-property} to get the existence of a $T>0$ such that
\begin{equation}
\label{eq:descent-at-T}
\varphi(z(T)) < \sup_{t \in [0,T]} \varphi(z(t)) \leq \varphi(z(0)) = \varepsilon.
\end{equation}
By uniform convergence of $u(\cdot+s_{i_j})$ to $z(\cdot)$, it follows for $j$ sufficiently large that
$$\sup_{t \in [0,T]} |\varphi(u(t+s_{i_j})) - \varphi(z(t))| < \varepsilon,$$
so
$$\sup_{t \in [0,T]} \varphi(u(t+s_{i_j})) \leq  \sup_{t \in [0,T]} |\varphi(u(t+s_{i_j})) - \varphi(z(t))| + \sup_{t \in [0,T]} \varphi(z(t)) \leq 2\varepsilon.$$
Therefore it must follow that 
\begin{equation}
\label{eq:statement-to-be-contradicted}
s_{e_{i_j}} > s_{i_j} + T
\end{equation}
for $j$ sufficiently large. We now find a contradiction to the statement \eqref{eq:statement-to-be-contradicted}. This is done by observing the sequence
$\ell_j := \max \{ \ell: s_{i_j} \leq s_\ell \leq s_{i_j} + T\}.$
From \eqref{eq:descent-at-T}, we have that there exists a $\delta > 0$ such that $\varphi(z(T)) \leq \varepsilon - 2 \delta.$ Observe that \begin{align*}
\lVert u_{\ell_j} - u(T+s_{i_j}) \rVert = \lVert u(s_{\ell_j}) - u(T+s_{i_j})\rVert \leq \lVert u_{\ell_j} - u_{\ell_j + 1}\rVert \rightarrow 0 \quad \text{ as } j \rightarrow \infty.                                                                                                                                                                                                                                                                           \end{align*}
Therefore $u_{\ell_j} \rightarrow u(T+s_{i_j})$ and hence $u_{\ell_j} \rightarrow z(T)$ as $j \rightarrow \infty$. By continuity, we get $\lim_{j \rightarrow \infty} \varphi(u_{\ell_j}) = \varphi(z(T)).$ Thus $\varphi(u_{\ell_j}) < \varepsilon - \delta$ for $j$ sufficiently large, a contradiction to \eqref{eq:statement-to-be-contradicted}. 
\end{proof}

\begin{proposition}
\label{prop:non-escape-argument} 
The limit $\lim_{t \rightarrow \infty} \varphi(u(t))$ exists. 
\end{proposition}

\begin{proof}
W.l.o.g.~assume $\liminf_{t\rightarrow \infty}\varphi(u(t))=0$; this is possible by the fact that $j$ and $\eta$ are bounded below.  Choosing $\varepsilon>0$ such that $\varepsilon \notin \varphi(S^{-1}(0)),$ we have by Lemma~\ref{lemma:exit-lemma} that for $N$ sufficiently large, $u_n \in \mathcal{L}_{2\varepsilon}$ for all $n \geq N$. Since $\varepsilon$ can be chosen to be arbitrarily small, we conclude that $\lim_{t\rightarrow \infty} \varphi(u(t)) = 0.$
\end{proof}

\paragraph{Proof of Theorem~\ref{theorem:convergence-variance-reduced-stochastic-gradient-decreasing-steps}.}
The fact that $\{\varphi(u_n)\}$ converges follows from Proposition~\ref{prop:non-escape-argument} and Lemma~\ref{lemma:same-limits}.
Since $\{u_n \} \subset C$, it trivially follows that $\{ f(u_n)\}$ converges a.s. Let $\bar{u}$ be a limit point of $\{ u_n\}$ and suppose that $0 \notin S(\bar{u})$.  Let $\{ u_{n_k}\}$ be a subsequence converging to $\bar{u}$ and let $z(\cdot)$ be the limit of $\{u(\cdot+s_{n_k})\}$. Then, by Lemma~\ref{lemma:descent-property}, there exists a $T>0$ such that
\begin{equation}
\label{eq:assumption-to-be-contradicted-convergence}
\varphi(z(T)) < \sup_{t \in [0,T]} \varphi(z(t)) \leq \varphi(\bar{u}).
\end{equation}
However, it follows from Proposition~\ref{prop:non-escape-argument} that
$$\varphi(z(T)) =  \lim_{k \rightarrow \infty} \varphi(u(T+s_{n_k})) = \lim_{t\rightarrow \infty} \varphi(u(t)) = \varphi(\bar{u}),$$
which is a contradiction to \eqref{eq:assumption-to-be-contradicted-convergence}.

\section{Application to PDE-Constrained Optimization under Uncertainty}
\label{sec:numerical-experiments}
In this section, we apply the algorithm presented in Sect.~\ref{subsection:ODE-proof} to a nonconvex problem from PDE-constrained optimization under uncertainty. In Sect.~\ref{subsection:ModelProblem}, we set up the problem and verify conditions for convergence of the stochastic proximal gradient method. We show numerical experiments in Sect.~\ref{subsection:experiments}.

\subsection{Model Problem}
\label{subsection:ModelProblem}
We first introduce notation and concepts specific to our application; see 
\cite{Troeltzsch2009,Evans1998}.
Let $D \subset \R^d$, $d \leq 3$ be an open and bounded Lipschitz domain. The inner product between vectors $x, y \in \R^d$ is denoted by $x \cdot y = \sum_{i=1}^d x_i y_i$. For a function $v:\R^d \rightarrow \R$, let $\nabla v(x) = ({\partial v(x)}/{\partial x_1}, \dots, {\partial v(x)}/{\partial x_d})^\top$ denote the gradient and for $w: \R^d \rightarrow \R^d$, let $\nabla \cdot w(x) = {\partial w_1(x)}/{\partial x_1} + \cdots + {\partial w_d(x)}/{\partial x_d}$ denote the divergence. We define the Sobolev space $H^1(D)$ = \{$u\in L^2(D)$ \text{ having weak derivatives } ${\partial u}/{\partial x_i} \in L^2(D)$, $i =1, \dots, d$\} and the closure of $C_c^\infty(D)$ in $H^1(D)$ by $H_0^1(D)$.  

We will focus on a semilinear diffusion-reaction equation with uncertainties, which describes transport phenomena at equilibrium and is motivated by \cite{Nouy2018}.  We assume that there exist random fields $a: D \times \Omega \rightarrow \R$ and $r: D \times \Omega \rightarrow \R$, which are the diffusion and reaction coefficients, respectively. To facilitate simulation, we will make a standard finite-dimensional noise assumption, meaning the random field has the form 
  $$a(x,\omega) = a(x,\xi(\omega)), \quad r(x,\omega) = r(x,\xi(\omega)) \quad \text{ in } D \times \Omega,$$
where $\xi(\omega) = (\xi_1(\omega), \dots, \xi_m(\omega))$ is a vector of real-valued uncorrelated random variables $\xi_{i}:\Omega \rightarrow \Xi_i \subset\R$. The support of the random vector will be denoted by $\Xi := \prod_{i=1}^m \Xi_i$. We consider the following PDE constraint, to be satisfied for almost every $\xi \in \Xi$:
\begin{equation}
 \label{eq:semilinear-PDE}
 \begin{aligned}
     -   \nabla \cdot  (a(x,\xi) \nabla y(x,\xi)) + r(x,\xi) (y(x,\xi))^3 &= u(x), \qquad (x,\xi) \in D \times \Xi, \\
 y(x,\xi) &= 0, \phantom{tex}\qquad (x,\xi) \in \partial D \times \Xi.\\
 \end{aligned} 
\end{equation} 

Optimal control problems with semilinear PDEs involving random coefficients have been studied in, for instance, \cite{Kouri2016,Kouri2019a}. 
We include a nonsmooth term as in \cite{Reyes2015} with the goal of obtaining sparse solutions. In the following, we assume that $\lambda_1 \geq 0$, $\lambda_2 \geq 0$, and  $y_D \in L^2(D)$. The model problem we solve is given by
\begin{equation}
 \label{eq:model-problem-semilinear}\tag{P'}
 \begin{aligned}
   &\min_{u \in C} \quad \left\lbrace\varphi(u):=  \frac{1}{2} \E [ \lVert y(\xi) - y_D\rVert_{L^2(D)}^2 ] + \frac{\lambda_2}{2}   \lVert u \rVert_{L^2(D)}^2 +\lambda_1 \lVert u \rVert_{L^1(D)}\right\rbrace   \\
 &  \quad \text{s.t.} \quad -   \nabla \cdot  (a(x,\xi) \nabla y) + r(x,\xi) y^3 = u(x), \qquad (x,\xi) \in D \times \Xi, \\
& \quad \phantom{\text{s.t.  } \quad -   \nabla \cdot  (a(x,\xi) \nabla y) + R(x,\xi)} y = 0, \phantom{tex}\qquad (x,\xi) \in \partial D \times \Xi,\\
&\quad \quad \quad \quad \quad C:= \{ u \in L^2(D):   \,u_a(x)  \leq u(x) \leq u_b(x)\,\,\text{ a.e. } x\in D\}.
 \end{aligned} 
\end{equation}
The following assumptions will apply in this section. In particular, we do not require uniform bounds on the coefficient $a(\cdot,\xi)$, which allow for modeling with log-normal random fields.
\begin{assumption}
\label{assumption:bilinearform}
We assume $y_D \in L^2(D)$, $u_a, u_b \in L^2(D)$, and $u_a \leq u_b$. There exist $a_{\min}(\cdot), a_{\max}(\cdot)$ such that $0< a_{\min}(\xi) < a(\cdot,\xi) < a_{\max}(\xi)< \infty$ in $D$ a.s.~and $a_{\min}^{-1}, a_{\max} \in L^p(\Xi)$ for all $p \in [1,\infty)$. Furthermore, there exists $r_{\max}(\cdot)$ such that $0 \leq r(\cdot,\xi) \leq r_{\max}(\xi)<\infty$ a.s.~and $r_{\max} \in L^p(\Xi)$ for all $p \in [1,\infty)$. 
\end{assumption}

Existence of a solution to Problem~\eqref{eq:model-problem-semilinear} follows by applying \cite[Proposition 3.1]{Kouri2019a}. The following result holds by \cite[Proposition 2.1]{Kouri2019a} combined with standard a priori estimates for a fixed realization $\xi$ to obtain \eqref{eq:bounds-y} and  \eqref{eq:continuous-dependence-y}.

\begin{lemma}
\label{lemma:well-posedness-PDE}
For almost every $\xi \in \Xi$,  \eqref{eq:semilinear-PDE} has a unique solution $y(\xi)=y(\cdot,\xi) \in H_0^1(D)$ and there exists a positive random variable $C_1 \in L^p(\Xi)$ for all $p \in [1,\infty)$ independent of $u$ such that for almost every $\xi \in \Xi$,
\begin{equation}
\label{eq:bounds-y}
 \lVert y(\xi) \rVert_{L^2(D)} \leq C_1(\xi) \lVert u \rVert_{L^2(D)}.
\end{equation}
Additionally, for $y_1(\xi)$ and $y_2(\xi)$ solving  \eqref{eq:semilinear-PDE} with $u=u_1$ and $u=u_2$, respectively, we have for almost every $\xi \in \Xi$ that
\begin{equation}
 \label{eq:continuous-dependence-y}
 \lVert y_1(\xi) - y_2(\xi) \rVert_{L^2(D)} \leq C_1(\xi) \lVert u_1 - u_2 \rVert_{L^2(D)}.
\end{equation}
\end{lemma}

By Lemma~\ref{lemma:well-posedness-PDE}, the control-to-state operator $T(\xi):L^2(D) \rightarrow H_0^1(D), u \mapsto T(\xi)u$ is well-defined for almost every $\xi$ and all $u \in L^2(D)$. Additionally, for almost every $\xi \in \Xi$, this mapping is in fact continuously Fr\'echet differentiable; this can be argued by verifying \cite[Assumption 1.47]{Hinze2009} as in \cite[pp.~76-78]{Hinze2009}. With that, we define the reduced functional $J:L^2(D) \times \Xi \rightarrow \R$ by $J(u,\xi):= \frac{1}{2} \lVert T(\xi)u - y_D\rVert_{L^2(D)}^2 + \frac{\lambda_2}{2} \lVert u \rVert_{L^2(D)}^2$ and we can define the stochastic gradient.

\begin{proposition}
\label{proposition:random-J-is-differentiable}
$J:L^2(D)\times\Xi \rightarrow \R$ is continuously Fr\'echet differentiable and the stochastic gradient is given by 
\begin{equation}
\label{eq:stochastic-gradient-application}
G(u,\xi) := \lambda_2 u - p(\cdot,\xi),
\end{equation}
where, given a solution $y = y(\cdot,\xi)$ to  \eqref{eq:semilinear-PDE}, the function $p = p(\cdot,\xi) \in H_0^1(D)$ is the solution to the adjoint equation
\begin{equation}
\label{eq:adjoint-equation}
\begin{aligned}
-\nabla \cdot (a(x,\xi) \nabla p) + 3 r(x,
\xi)y^2 p &= y_D-y, \quad (x,\xi) \in D \times \Xi \\
p &=0,\phantom{- y_D} \quad \phantom{t}(x,\xi) \in \partial D \times \Xi.
\end{aligned}
\end{equation}
Furthermore, for almost every $\xi \in \Xi$, with the same $C_1 \in L^p(\Xi)$ for all $p \in [1,\infty)$ as in Lemma~\ref{lemma:well-posedness-PDE},
\begin{equation}
\label{eq:a-priori-p}
 \lVert p(\cdot,\xi) \rVert_{L^2(D)} \leq C_1(\xi) \lVert y_D - y(\xi)\rVert_{L^2(D)}.
\end{equation}
Additionally, for $p_1(\xi)$ and $p_2(\xi)$ solving \eqref{eq:adjoint-equation} with $y=y_1(\xi)$ and $y=y_2(\xi)$, respectively (where $y_i(\xi)$ solves   \eqref{eq:semilinear-PDE} with $u=u_i$),
\begin{equation}
\label{eq:continuous-dependence-p}
 \lVert p_1(\xi) - p_2(\xi) \rVert_{L^2(D)} \leq C_1(\xi) \lVert y_1(\xi) - y_2(\xi) \rVert_{L^2(D)}.
\end{equation}
\end{proposition}

The proofs of the above and following proposition are in Sect.~\ref{subsection:auxiliary-proofs-application}.
We define $j:L^2(D) \rightarrow \R$ by $j(u):= \E[J(u,\xi)]$ for all $u\in L^2(D)$ and show that it is continuously Fr\'echet differentiable in the following proposition.
\begin{proposition}
 \label{proposition:j-is-differentiable}
The function $j:L^2(D) \rightarrow \R$ is continuously Fr\'echet differentiable and $\E[G(u,\xi)] = \nabla j(u)$ for all $u\in L^2(D)$.
\end{proposition}

Now, we present the main result of this section, which is the verification of assumptions for the convergence of Algorithm~\ref{alg:PSG_Hilbert_Nonconvex_Decreasing_Steps}.

\begin{theorem}
Problem \eqref{eq:model-problem-semilinear} satisfies Assumption~\ref{asu1ii} as well as Assumption~\ref{asu4i}--Assumption~\ref{asu4iv}.
\end{theorem}

\begin{proof}
For Assumption~\ref{asu1ii}, we note that by Proposition~\ref{proposition:j-is-differentiable}, $j$ is continuously Fr\'echet differentiable and $\E[G(u,\xi)] = \nabla j(u)$ for all $u \in L^2(D)$. Now, for arbitrary $u_1,u_2 \in L^2(D)$, we have by Jensen's inequality, \eqref{eq:stochastic-gradient-application}, and H\"older's inequality applied to \eqref{eq:continuous-dependence-p} and \eqref{eq:continuous-dependence-y} that
\begin{align*}
 &\lVert \nabla j(u_1) - \nabla j(u_2)\rVert_{L^2(D)} \leq \E[\lVert G(u_1,\xi)-G(u_2,\xi)\rVert_{L^2(D)}]\\
 & \quad\leq \lambda_2 \lVert u_1 - u_2 \rVert_{L^2(D)} + \E[\lVert p_1(\xi) - p_2(\xi)\rVert_{L^2(D)} ]\\
 & \quad\leq \lambda_2 \lVert u_1 - u_2 \rVert_{L^2(D)} + \left(\E[(C_1(\xi))^2]\right)^{1/2}\left(\E[\lVert y_1(\xi) - y_2(\xi)\rVert_{L^2(D)}^2 ]\right)^{1/2}\\
 & \quad\leq \lambda_2 \lVert u_1 - u_2 \rVert_{L^2(D)} + \lVert C_1\rVert_{L^2(\Xi)}^2 \lVert u_1 - u_2\rVert_{L^2(D)}.
\end{align*}
Since $\lVert C_1\rVert_{L^2(\Xi)}^2 < \infty$ it follows that $j \in C^{1,1}_L(L^2(D))$.

Assumption~\ref{asu4i} is obviously satisfied.  For Assumption~\ref{asu4ii}, we have that the function $\eta(u)=\lambda_1 \lVert u\rVert_{L^1(D)} \in \Gamma_0(L^2(D))$ and is clearly bounded below; additionally, $\eta$ is globally Lipschitz and therefore satisfies \eqref{eq:local-Lipschitz-bound-h}. 
For Assumption~\ref{asu4iii}, we have by \eqref{eq:stochastic-gradient-application}, \eqref{eq:a-priori-p}, and \eqref{eq:bounds-y} the bound
\begin{equation}
\label{eq:stochastic-gradient-bound-example}
\lVert G(u,\xi) \rVert_{L^2(D)} \leq \lambda_2 \lVert u \rVert_{L^2(D)} + C_1(\xi) \lVert y_D \rVert_{L^2(D)} + (C_1(\xi))^2 \lVert u\rVert_{L^2(D)}
\end{equation}
and furthermore $\E[\lVert G(u,\xi) \rVert_{L^2(D)}^2] =:M(u)<\infty$ by integrability of $\xi \mapsto C_1(\xi)$. Assumption~\ref{asu4iv} follows for any $u \in C$ (and hence any convergent sequence $\{u_n\}$ in $C$) by \eqref{eq:stochastic-gradient-bound-example}.
\end{proof}

The last assumption from Assumption~\ref{assumptions:general-convergence-proof} is technical and difficult to verify for general functions in infinite dimensions. Indeed, \cite{Kupka1965} gave an example of a $C^\infty$-function whose critical values make up a set of measure greater than zero. In finite dimensions the story is easier: the Morse--Sard theorem guarantees that Assumption~\ref{asu4v} holds if $f:\R^n \rightarrow \R$ and $f \in C^k$ for $k \geq n$. In infinite dimensions, certain well-behaved functions, in particular Fredholm operators, see \cite{Smale2000}, satisfy this assumption.

\subsection{Numerical Experiments}
\label{subsection:experiments}
In this section, we demonstrate Algorithm~\ref{alg:PSG_Hilbert_Nonconvex_Decreasing_Steps} on Problem \eqref{eq:model-problem-semilinear}. Simulations were run using FEniCS by \cite{Alnes2015} on a laptop with Intel Core i7 Processor (8 x 2.6 GHz) with 16 GB RAM. Let the domain be given by $D=(0,1)\times(0,1)$ and the constraint set be given by {$C= \{ u \in L^2(D) \,|\, -0.5 \leq u(x) \leq 0.5  \,\, \forall x \in D\}.$} We modify \cite[Example 6.1]{Reyes2015}, with $y_D(x)=\sin(2 \pi x_1)\sin (2\pi x_2) \exp(2 x_1)/6$, $\lambda_1 = 0.008$, and $\lambda_2 = 0.001.$ We generate random fields using a Karhunen-Lo\`eve expansion, with means $a_0 = 0.5$ and $r_0 = 0.5$, number of summands $m = 20$, and $\xi^{a,i},\xi^{r,i} \sim U(-\sqrt{0.5},\sqrt{0.5})$, 
where $U(a,b)$ denotes the uniform distribution between real numbers $a$ and $b$, $a<b$. The eigenfunctions and eigenvalues are given by 
$$\tilde{\phi}_{j,k}(x):= 2\cos(j \pi x_2)\cos(k \pi x_1), \quad \tilde{\lambda}_{k,j}:=\frac{1}{4} \exp(-\pi(j^2+k^2)l^2), \quad j,k \geq 1,$$
where we reorder terms so that the eigenvalues appear in descending order (i.e., $\phi_1 = \tilde{\phi}_{1,1}$ and $\lambda_1 = \tilde{\lambda}_{1,1}$) and we choose correlation length $l=0.5$. Thus
\begin{equation}
\label{eq:random-field-expansion}
a(x,\xi) = a_0 + \sum_{i=1}^m \sqrt{\lambda_i} \phi_i  \xi^{a,i}, \quad r(x,\xi) = r_0 + \sum_{i=1}^m \sqrt{\lambda_i} \phi_i  \xi^{r,i}.
\end{equation}
For Algorithm~\ref{alg:PSG_Hilbert_Nonconvex_Decreasing_Steps}, we generate samples with $\xi_n = (\xi_n^{a,1}, \dots, \xi_n^{a,m}, \xi_n^{r,a}, \dots,\xi_n^{r,m})$ at each iteration $n$. The step size is chosen to be $t_n = \theta/n$ with $\theta = 100$, where the scaling was chosen such that $\theta \approx 1/\lVert G(u_1, \xi_1)\rVert$. The initial point was $u_1(x) = \sin(4\pi x_1) \sin (4 \pi x_2).$

A uniform mesh $\mathcal{T}$ with 9800 shape regular triangles $T$ was used. We denote the mesh fineness with $\hat{h} = \max_{T \in \mathcal{T}}\operatorname{diam}(T)$.
The state and adjoint were discretized using piecewise linear finite elements, (where $\mathcal{P}_i$ denotes the space of polynomials of degree up to $i$), given by the set
\begin{align*}
V_{\hat{h}} &:= \lbrace v \in H_0^1(D): v|_{T} \in \mathcal{P}_1(T) \text{  for all } T\in \mathcal T \rbrace.
\end{align*}
For the controls, we choose a discretization of $L^2(D)$ by piecewise constants, given by the set 
\begin{align*}
U_{\hat{h}} &:= \lbrace u \in L^2(D): v|_{T} \in \mathcal{P}_0(T) \text{  for all } T\in \mathcal{T} \rbrace,\quad
C_{\hat{h}} := U_{\hat{h}}\cap C.
\end{align*}
We use the $L^2$-projection $P_{\hat{h}}\colon L^2(D) \rightarrow U_{\hat{h}}$ defined for each $v \in L^2(D)$ by
\[
 P_{\hat{h}}(v)\bigl\lvert_T := \frac{1}{|T|}\int_T v\,\mathrm{d}x.
\]
This is done to project the stochastic gradient onto the $L^2(D)$ space as in \cite{Geiersbach2020b}. Hence, the last line of Algorithm~\ref{alg:PSG_Hilbert_Nonconvex_Decreasing_Steps} is given by the expression $u_{n+1}:=\prox_{t_n h}\left( u_n - t_nP_{\hat{h}} G(u_n,\xi_{n})\right).$
For the computation of the proximity operator $\prox_{t(\eta+\delta_C)}(z) = \argmin_{-0.5 \leq v \leq 0.5} \{ \lambda_1 \lVert v \rVert_{L^1(D)} + \frac{1}{2t} \lVert v - z \rVert_{L^2(D)}^2\}$, we use the formula from \cite[Example 6.22]{Beck2017}, defined piecewise on each element of the mesh. For each $T \in \mathcal{T}$, it is given by 
$$\prox_{t(\eta+\delta_C)}(z|_T) = \min \{ \max \{ |z|_T|-t \lambda_1,0\},0.5\}\textup{sgn}(z|_T).$$ 
For convergence plots, we use a heuristic to approximate the objective function and the measure of stationarity by increasing sampling as the control reaches stationarity. To be more precise, we use a sequence of sample sizes $\{m_n\}$ with $m_{ n} = 10\lfloor \tfrac{n}{50}\rfloor +1$ newly generated i.i.d.~samples $(\xi_{n,1}, \dots, \xi_{n,m_n})$ and compute 
\begin{align*}
\hat{f}_n &:= \frac{1}{m_n}\sum_{j=1}^{m_n} J(u_n, \xi_{n,j}) + \eta(u_n), \\
{r}_n &:= \left\lVert u_n - \prox_{\eta+\delta_C}\left(u_n - \frac{1}{m_n}\sum_{j=1}^{m_n}P_{\hat{h}} G(u_n,\xi_{n,j})\right)\right\rVert_{L^2(D)}.
\end{align*}
The algorithm is terminated for $n\geq 50$ if $\hat{r}_n := \sum_{k=n-50}^n {r}_n \leq \text{tol}$ with $\text{tol}=2e^{-4}$. The parameters for our heuristic termination rule were tuned, for illustration purposes only, so that the algorithm stopped after several hundred iterations.  A plot of the control after termination is shown in~Fig.~\ref{fig:control}. The effect of the sparse term $\eta$ as well as the constraint set $C$ can be seen clearly. Decay of the objective function value and the stationarity measure are shown in Fig.~\ref{fig:experiment1}. We see convergence of the objective function values and the stationarity measure tends to zero as expected. 

Additionally, we conduct an experiment to demonstrate mesh independence of the algorithm by running the algorithm once each for different meshes and comparing the number of iterations needed until the tolerance $\text{tol}$ is reached. In Table~\ref{table:mesh-independence}, we see that these iteration numbers are of the same order. The estimate for the objective function $\hat{f}_N$ is also included at the final iteration $N$, demonstrating how solutions become more exact on finer meshes.

\begin{figure}
 \centering
  \includegraphics[height = 4.0cm]{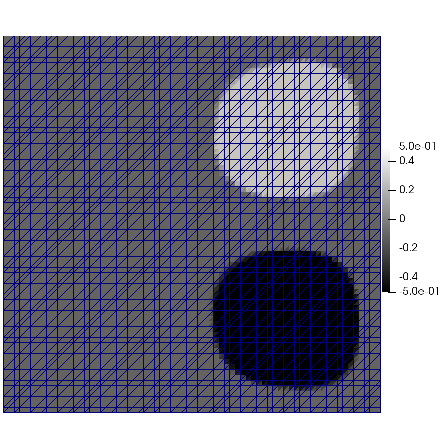}
  \caption{The control $u$ after 251 iterations}
      \label{fig:control}
\end{figure}

\begin{figure}
 
  \includegraphics[height = 4.0cm]{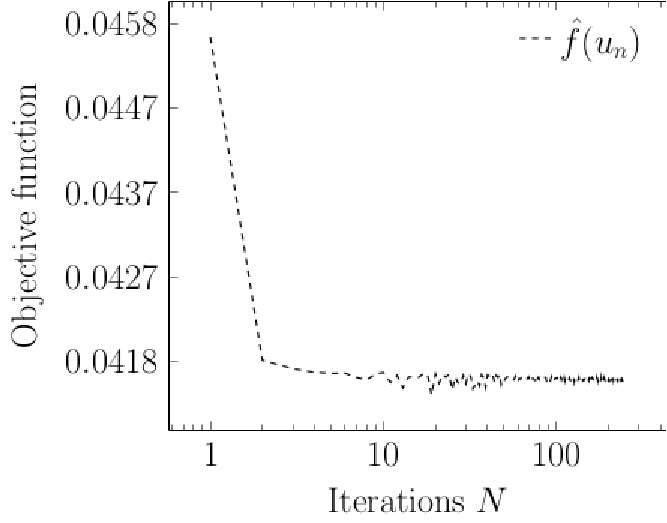}
    \includegraphics[height = 4.0cm]{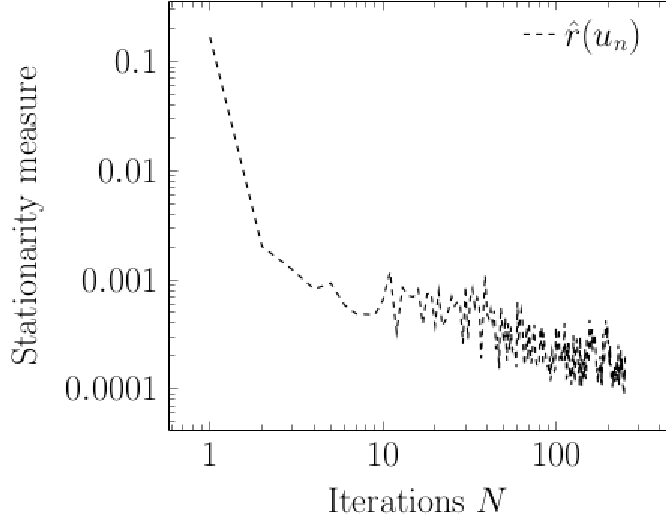}
  \caption{Behavior of the objective function (left) and the stationarity measure (right)}
      \label{fig:experiment1}
\end{figure}

 \begin{table}
   \begin{center}
  \begin{tabular}{| c | c | c |c |}
  \hline
  $\hat{h}$ & \# triangles & objective function $\hat{f}_N$ & \# iterations $N$ until $\hat{r}_N \leq$ tol \\
  \hline
  $  7.1e^{-2}$ & $800$ & $4.160e^{-2}$ & $191$  \\
  $ 4.7e^{-2}$ & $1800$ & $4.157e^{-2}$ & $295$\\
  $3.5e^{-2}$ & $3200$ & $4.157e^{-2}$ & $233$  \\
  $2.8e^{-2}$ & $5000$ & $4.156e^{-2}$ & $257$  \\  
 $2.4e^{-2}$ & $7200$ & $4.156e^{-2}$ & $271$  \\  
$2.0e^{-2}$ & $9800$ & $4.155e^{-2}$ & $251$  \\  
    \hline
  \end{tabular} 
    \end{center}
    \caption{Experiment showing mesh independence}
    \label{table:mesh-independence}
  \end{table}

\section{Conclusion}
\label{sec:conclusion}
In this paper, we presented asymptotic convergence analysis for two variants of the stochastic proximal gradient algorithm in Hilbert spaces. The main results address the asymptotic convergence to stationary points of general functions defined over a Hilbert space. Moreover, we presented an application to the theory in the form of a problem from PDE-constrained optimization under uncertainty. Assumptions for convergence were verified for a tracking-type problem with a $L^1$-penalty term subject to a semilinear elliptic PDE with random coefficients and box constraints. Numerical experiments demonstrated the effectiveness of the method.

The ODE method from Sect.~\ref{subsection:ODE-proof} allowed us to prove a more general result with weaker assumptions on the objective function. However, we needed to introduce an assumption on the set of critical values in the form of Assumption~\ref{asu4v}. While we did not verify this assumption for our model problem, it would be interesting to know whether this assumption is verifiable for this class of problems. We had to be slightly more restrictive on the nonsmooth term in Sect.~\ref{subsection:ODE-proof} than we were in Sect.~\ref{subsection:SPGM-Variance-Reduced}. The advantages in terms of computational cost of Algorithm~\ref{alg:PSG_Hilbert_Nonconvex_Decreasing_Steps} over Algorithm~\ref{alg:PSG_Hilbert_Nonconvex} are clear: the use of decreasing step sizes in Algorithm~\ref{alg:PSG_Hilbert_Nonconvex_Decreasing_Steps} means that increased sampling is not needed. Additionally, there is no need to determine the Lipschitz constant for the gradient, which in the application depends on (among other things) the Poincar\'e constant and the lower bound on the random fields, and thus lead to a prohibitively small constant step size. This phenomenon has been demonstrated in \cite{Geiersbach2020b}. 

How to scale the decreasing step size $t_n$ remains an open question. In practice, the scaling of the step size can be tuned offline. An improper choice of the scaling $c$ in the step size $t_n = c/n^\alpha$ for $0.5 < \alpha \leq 1$ can lead to arbitrarily slow convergence; this was demonstrated in \cite{Nemirovski2009}. While this was not the focus of our work, efficiency estimates for nonconvex problems might also be possible following the work by  \cite{Bottou2018,Lei2018,Ghadimi2016}. In lieu of efficiency estimates, it would be desirable to have better termination conditions that do not rely on increased sampling as our heuristic did in the numerical experiments.   
Finally, it would be natural to investigate mesh refinement strategies as in \cite{Geiersbach2020b}. For more involved choices of nonsmooth terms, the $\prox$ computation is also subject to numerical error and should be treated. 

\appendix
\section{Auxiliary Results}\label{subsection:auxiliary-proofs}
To prove Lemma~\ref{lemma:quadratic-variations-bounded-imply-convergence}, we first need the following result. 

\begin{proposition}
 \label{proposition:radon-nikodym-property}
For $1 \leq p \leq \infty$, every $H$-valued martingale that is bounded in the Bochner space $L^p(\Omega,H)$ converges a.s.
\end{proposition}
\begin{proof}
Since $H$ is a Hilbert space, it is reflexive and therefore has the Radon--Nikodym property by \cite[Corollary 2.11]{Pisier2016}. The rest of the proof can be found in \cite[Theorem 2.5]{Pisier2016}.
\end{proof}

\paragraph{Proof of Lemma~\ref{lemma:quadratic-variations-bounded-imply-convergence}}
\begin{proof}
It is straightforward to show that $\E[\lVert v_n\rVert^2] = \E[\lVert v_1 \rVert^2] + \sum_{k=1}^{n-1} \E[\lVert v_{k+1} - v_{k}\rVert^2],$
and therefore boundedness of $v_n$ for all $n$ follows from \eqref{eq:quadratic-variations-proof} and vice versa. Supposing now that \eqref{eq:quadratic-variations-proof} holds, the fact that $\{v_n\}$ converges to a limit $v_\infty$ follows by Proposition~\ref{proposition:radon-nikodym-property}.
\end{proof}

\section{Auxiliary Proofs for Application}\label{subsection:auxiliary-proofs-application}

\paragraph{Proof of Proposition~\ref{proposition:random-J-is-differentiable}}
\begin{proof}
Continuous differentiability of $J(\cdot,\xi):L^2(D)\rightarrow \R$ follows from continuous differentiability of $u \mapsto T(\xi)u$and the fact that $(u,y) \mapsto \tilde{J}(u,y) = \tfrac{1}{2} \lVert y-y_D\rVert_{L^2(D)}^2 + \tfrac{\lambda_2}{2}\lVert u\rVert_{L^2(D)}^2$ is continuously Fr\'echet differentiable. 
One obtains \eqref{eq:stochastic-gradient-application} and \eqref{eq:adjoint-equation} by fixing a realization $\xi \in \Xi$ and computing the derivative of $u \mapsto J(u,\xi)$ as in, e.g., \cite[p.~58-59]{Hinze2009}. Bounds \eqref{eq:a-priori-p} and \eqref{eq:continuous-dependence-p} follow from standard a priori estimates.
\end{proof}

\paragraph{Proof of Proposition~\ref{proposition:j-is-differentiable}}
\begin{proof}
We verify the conditions of Lemma~\ref{lemma:frechet-exchange-derivative-expectation} from Sect.~\ref{subsection:differentiability-expectation}. Fr\'echet differentiability of $J:L^2(D) \times \Xi \rightarrow \R$ for almost every $\xi$ follows from Proposition~\ref{proposition:random-J-is-differentiable}. The function $j$ is well-defined and finite-valued for all $u \in L^2(D)$, since 
\begin{align*}
 j(u) &= \frac{1}{2}\E[\lVert y - y_D \rVert^2_{L^2(D)}] + \frac{\lambda_2}{2} \lVert u \rVert_{L^2(D)}^2 \leq \E[\lVert T(\xi)u \rVert_{L^2(D)}^2] + \lVert y_D\rVert_{L^2(D)}^2 + \frac{\lambda_2}{2} \lVert u \rVert_{L^2(D)}^2
\end{align*}
is finite by $T(\xi)u = y(\xi)$ and \eqref{eq:bounds-y} along with the assumption that $y_D \in L^2(D)$.
Now, for every $v \in C$, there exists a $y_v(\xi)$ satisfying  \eqref{eq:semilinear-PDE} with $u=v$ and a $p_v(\xi)$ satisfying \eqref{eq:adjoint-equation} with $y=y_v(\xi)$. Thus by \eqref{eq:a-priori-p} followed by \eqref{eq:bounds-y},
\begin{align*}
 \lVert G(v,\xi) \rVert_{L^2(D)} &= \lVert \lambda_2v - p_v(\xi) \rVert_{L^2(D)} \leq \lambda_2 \lVert v \rVert_{L^2(D)} + C_1(\xi)\lVert y_D - y_v(\xi)\rVert_{L^2(D)}\\
 & \leq \lambda_2 \lVert v \rVert_{L^2(D)} + C_1(\xi) \lVert y_D \rVert_{L^2(D)} + (C_1(\xi))^2 \lVert v\rVert_{L^2(D)}=:C(\xi).
\end{align*}
Notice that $C \in L^p(\Xi)$ for all $p \in [1,\infty)$ by nature of the mapping $\xi \mapsto C_1(\xi)$. 
Therefore, the conditions of Lemma~\ref{lemma:frechet-exchange-derivative-expectation} are satisfied and we have proven Fr\'echet differentiability of $j$.
\end{proof}

\section{Differentiability of Expectation Functionals}
\label{subsection:differentiability-expectation}Let $(X, \lVert \cdot \rVert_X)$ be a Banach space and let $J: X \times \Omega \rightarrow \R$ be a random variable functional. We summarize under what conditions we can exchange the integral and the derivative for the functional $j: X \rightarrow \R$, where $j(u) = \int_\Omega J(u,\omega) \D \pP(\omega)$. 

The following definition gives the minimal requirement for exchanging the derivative and expectation, namely, requiring  $J:X \times \Omega \rightarrow \R$ to be $L^1$-Fr\'echet differentiable.
\begin{definition}
\label{definition-Lp-differentiable}
A $p$-times integrable random functional $J:X \times \Omega \rightarrow \R$ is called $L^p$-Fr\'echet differentiable at $u$ if for an open set $U \subset X$ there exists a bounded and linear random operator $A:U \times \Omega \rightarrow \R$ such that $\lim_{h \rightarrow 0} \lVert J_{\omega}(u + h) - J_{\omega}(u) + A(u,\omega)h \rVert_{L^p(\Omega)} / \lVert h \rVert_X = 0$.
\end{definition}
By H\"older's inequality, if $u \mapsto J(u,\cdot)$ is $L^p$-differentiable and $1 \leq r < p$, then it is also $L^r$-differentiable with the same derivative. This implies that $j:X \rightarrow \R$ is Fr\'echet differentiable at $u$.

The condition in \ref{definition-Lp-differentiable} might be difficult to verify directly. For this reason, we consider other assumptions on an open neighborhood $U$ of $X$ containing $u$. We denote the functional $J(\cdot,\omega):X \rightarrow \R$ for a fixed realization $\omega \in \Omega$ by $J_\omega:X \rightarrow \R$.

\begin{assumption}
\label{asu-expectation}
\subasu \label{subasu-expectation1} The expectation $j(v)$ is well-defined and finite-valued for all $v \in U.$\\
\subasu \label{subasu-expectation2} For almost every $\omega \in \Omega$, the functional $J_\omega:X \rightarrow \R$ is Fr\'echet differentiable at $u$. Moreover, there exists a positive random variable $C(\cdot) \in L^1(\Omega)$ such that for all $v \in U$ and almost every $\omega \in \Omega$,
 \begin{equation}
  \label{eq:randomLipschitz}
\lVert J'_\omega(v) \rVert_{X^*} \leq C(\omega).
 \end{equation} 
\end{assumption}

\begin{lemma}
\label{lemma:frechet-exchange-derivative-expectation}
Suppose Assumption~\ref{asu-expectation} holds. Then $j$ is Fr\'echet differentiable at $u$ and
\begin{equation}
\label{eq:expectation-frechet-ch3}
j'(u) = \E[J'_\omega(u)]. 
\end{equation}
\end{lemma}
\begin{proof}
By the mean value theorem, for $h$ close enough to $u$,  there exists a $z$ within the neighborhood containing $u+h$ and $u$ that satisfies
$| J_{\omega}(u + h) - J_{\omega}(u) | \leq \lVert J'_\omega(z)\rVert_{X^*} \rVert h\rVert_{X}.$
Now, we have for almost every $\omega \in \Omega$ that
\begin{equation*}
 \begin{aligned}
  \frac{|J_{\omega}(u+h) - J_{\omega}(u) - J'_{\omega}(u)h|}{\lVert h \rVert_X} &\leq  \frac{|J_{\omega}(u+h) - J_{\omega}(u) |}{\lVert h \rVert_X} +  \frac{|J'_{\omega}(u)h|}{\lVert h \rVert_X}\\
  & \leq \lVert J'_{\omega}(z) \rVert_{X^*} + \lVert J'_{\omega}(u) \rVert_{X^*} \leq 2 C(\omega).
 \end{aligned}
\end{equation*}
By Assumption~\ref{subasu-expectation2}, $C(\cdot)$ is integrable, so by Lebesgue's dominated convergence theorem, it follows that
\begin{equation}
 \label{eq:switch-expectation-proof}
\begin{aligned}
&\lim_{h \rightarrow 0}  \frac{\int_{\Omega} |J_{\omega}(u+h) - J_{\omega}(u) - J'_{\omega}(u)h| \D \pP(\omega)}{\lVert h \rVert_X} \\
& \quad \quad=  \int_{\Omega} \lim_{h \rightarrow 0} \frac{ |J_{\omega}(u+h) - J_{\omega}(u) - J'_{\omega}(u)h|}{\lVert h \rVert_X}  \D \pP(\omega) = 0,
\end{aligned}
\end{equation}
where the last equality follows by Assumption~\ref{subasu-expectation2}. 
Now consider the mapping {$F: h \mapsto \int_{\Omega} J_{\omega}'(u)h \D \pP(\omega).$}  It is straightforward to show that this is a bounded and linear operator. Therefore, we use Assumption~\ref{subasu-expectation1} to get
\begin{align*}
&\lim_{h \rightarrow 0}  \frac{|\int_{\Omega} J_{\omega}(u+h) \D \pP(\omega) - \int_{\Omega}J_{\omega}(u) \D \pP(\omega) - F(h)| }{\lVert h \rVert_X} \\
&\quad \quad= \lim_{h \rightarrow 0} \frac{|\int_{\Omega} (J_{\omega}(u+h) - J_{\omega}(u) - J'_{\omega}(u)h) \D \pP(\omega)|}{\lVert h \rVert_X} = 0,
\end{align*}
where the second equality holds by the triangle inequality and 
and \eqref{eq:switch-expectation-proof}. Therefore $j$ is Fr\'echet differentiable at $u$ with derivative $F = \int_{\Omega} J'_{\omega}(u) \D \pP(\omega)$.
\end{proof}

\newenvironment{acknowledgements} {\renewcommand\abstractname{Acknowledgements}\begin{abstract}} {\end{abstract}} 
\begin{acknowledgements}
The authors would like to thank the anonymous reviewers for their thoughtful comments and efforts towards improving our manuscript. This is a preprint of an article published in Computational Optimization and Applications. The final authenticated version is available online at: https://doi.org/10.1007/s10589-020-00259-y.
\end{acknowledgements}

\bibliographystyle{plain}
\bibliography{references}

\begin{thebibliography}{10}

\bibitem{Ali2017}
Ahmad~Ahmad Ali, Elisabeth Ullmann, and Michael Hinze.
\newblock Multilevel {M}onte {C}arlo analysis for optimal control of elliptic
  {PDE}s with random coefficients.
\newblock {\em SIAM/ASA Journal on Uncertainty Quantification}, 5(1):466--492,
  2017.

\bibitem{Alnes2015}
Martin~S. Aln{\ae}s, Jan Blechta, Johan Hake, August Johansson, Benjamin
  Kehlet, Anders Logg, Chris Richardson, Johannes Ring, Marie~E. Rognes, and
  Garth~N. Wells.
\newblock The {FEniCS} project version 1.5.
\newblock {\em Archive of Numerical Software}, 3(100), 2015.

\bibitem{Barty2007}
Kengy Barty, Jean-S{\'e}bastien Roy, and Cyrille Strugarek.
\newblock Hilbert-valued perturbed subgradient algorithms.
\newblock {\em Mathematics of Operations Research}, 32(3):551--562, 2007.

\bibitem{Bauschke2011}
Heinz~H Bauschke, Patrick~L Combettes, et~al.
\newblock {\em Convex Analysis and Monotone Operator Theory in Hilbert Spaces},
  volume 408.
\newblock Springer, 2011.

\bibitem{Beck2017}
Amir Beck.
\newblock {\em First-order Methods in Optimization}, volume~25.
\newblock SIAM, 2017.

\bibitem{Bottou1998}
L{\'e}on Bottou.
\newblock Online learning and stochastic approximations.
\newblock {\em On-Line Learning in Neural Networks}, 17(9):142, 1998.

\bibitem{Bottou2018}
L\'eon Bottou, Frank Curtis, and Jorge Nocedal.
\newblock Optimization methods for large-scale machine learning.
\newblock {\em SIAM Review}, 60(2):223--311, 2018.

\bibitem{Brezis1973}
H.~Br\'{e}zis.
\newblock {\em Op\'{e}rateurs Maximaux Monotones et Semi-Groupes de
  Contractions dans les Espaces de {H}ilbert}.
\newblock North-Holland Publishing Co., Amsterdam-London; American Elsevier
  Publishing Co., Inc., New York, 1973.
\newblock North-Holland Mathematics Studies, No.~5. Notas de Matem\'{a}tica
  (50).

\bibitem{Cazenave1998}
Thierry Cazenave, Andrea Braides, and Alain Haraux.
\newblock {\em An Introduction to Semilinear Evolution Equations}, volume~13.
\newblock Oxford University Press on Demand, 1998.

\bibitem{Chen2002}
Xiaohong Chen and Halbert White.
\newblock Asymptotic properties of some projection-based {R}obbins-{M}onro
  procedures in a {H}ilbert space.
\newblock {\em Studies on Nonlinear Dynamics and Econometrics}, 6:1--53, 2002.

\bibitem{Clarke1990}
Frank~H Clarke.
\newblock {\em Optimization and Nonsmooth Analysis}, volume~5.
\newblock {SIAM}, 1990.

\bibitem{Culioli1990}
J.-C. Culioli and G.~Cohen.
\newblock Decomposition/coordination algorithms in stochastic optimization.
\newblock {\em SIAM Journal on Control and Optimization}, 28(6):1372--1403,
  1990.

\bibitem{Davis2018}
Damek Davis, Dmitriy Drusvyatskiy, Sham Kakade, and Jason~D Lee.
\newblock Stochastic subgradient method converges on tame functions.
\newblock {\em Foundations of Computational Mathematics}, pages 1--36, 2018.

\bibitem{Reyes2015}
Juan~Carlos De~los Reyes.
\newblock {\em Numerical PDE-Constrained Optimization}.
\newblock Springer, 2015.

\bibitem{Duchi2018}
John~C Duchi and Feng Ruan.
\newblock Stochastic methods for composite and weakly convex optimization
  problems.
\newblock {\em SIAM Journal on Optimization}, 28(4):3229--3259, 2018.

\bibitem{Duflo2013}
Marie Duflo.
\newblock {\em Random Iterative Models}, volume~34.
\newblock Springer Science \& Business Media, 2013.

\bibitem{Ermoliev1969}
Yu~M Ermoliev.
\newblock On the stochastic quasi-gradient method and stochastic quasi-feyer
  sequences.
\newblock {\em Kibernetika}, 2:72--83, 1969.

\bibitem{Evans1998}
Lawrence Evans.
\newblock {\em Partial Differential Equations}, volume 19, Graduate Studies in
  Mathematics.
\newblock American Mathematical Society, Providence, R.I., 1998.

\bibitem{Geiersbach2019}
Caroline Geiersbach and Georg~Ch Pflug.
\newblock Projected stochastic gradients for convex constrained problems in
  {H}ilbert spaces.
\newblock {\em SIAM Journal on Optimization}, 29(3):2079--2099, 2019.

\bibitem{Geiersbach2020b}
Caroline Geiersbach and Winnifried Wollner.
\newblock A stochastic gradient method with mesh refinement for pde-constrained
  optimization under uncertainty.
\newblock {\em SIAM Journal on Scientific Computing}, 42(5):A2750--A2772, 2020.

\bibitem{Ghadimi2016}
Saeed Ghadimi, Guanghui Lan, and Hongchao Zhang.
\newblock Mini-batch stochastic approximation methods for nonconvex stochastic
  composite optimization.
\newblock {\em Mathematical Programming}, 155(1-2):267--305, 2016.

\bibitem{Goldstein1988}
Larry Goldstein.
\newblock Minimizing noisy functionals in {H}ilbert space: An extension of the
  {K}iefer-{W}olfowitz procedure.
\newblock {\em Journal of Theoretical Probability}, 1(2), 1988.

\bibitem{Hinze2009}
Michael Hinze, Rene Pinnau, Michael Ulbrich, and Stefan Ulbrich.
\newblock {\em Optimization with PDE Constraints}.
\newblock Springer, 2009.

\bibitem{Keshavarzzadeh2017}
Vahid Keshavarzzadeh, Felipe Fernandez, and Daniel~A Tortorelli.
\newblock Topology optimization under uncertainty via non-intrusive polynomial
  chaos expansion.
\newblock {\em Computer Methods in Applied Mechanics and Engineering},
  318:120--147, 2017.

\bibitem{Kiefer1952}
Jack Kiefer and Jacob Wolfowitz.
\newblock Stochastic estimation of the maximum of a regression function.
\newblock {\em The Annals of Mathematical Statistics}, 23(3):462--466, Sep
  1952.

\bibitem{Kouri2016}
D.~P. Kouri and T.~M. Surowiec.
\newblock Risk-averse {PDE}-constrained optimization using the conditional
  value-at-risk.
\newblock {\em SIAM Journal on Optimization}, 26(1):365--396, 2016.

\bibitem{Kouri2013}
D.P. Kouri, M.~Heinkenschloss, D.~Ridzal, and B.G. Van~Bloemen Waanders.
\newblock A trust-region algorithm with adaptive stochastic collocation for
  {PDE} optimization under uncertainty.
\newblock {\em SIAM Journal on Scientific Computing}, 35(4):A1847--A1879, 2013.

\bibitem{Kouri2019a}
Drew Kouri and Thomas Surowiec.
\newblock Risk-averse optimal control of semilinear elliptic pdes.
\newblock 2019.

\bibitem{Kouri2014a}
Drew~P Kouri.
\newblock A multilevel stochastic collocation algorithm for optimization of
  {PDE}s with uncertain coefficients.
\newblock {\em SIAM/ASA Journal on Uncertainty Quantification}, 2(1):55--81,
  2014.

\bibitem{Kunoth2016}
Angela Kunoth and Christoph Schwab.
\newblock Sparse adaptive tensor {G}alerkin approximations of stochastic
  {PDE}-constrained control problems.
\newblock {\em SIAM/ASA Journal on Uncertainty Quantification},
  4(1):1034--1059, 2016.

\bibitem{Kupka1965}
Ivan Kupka.
\newblock Counterexample to the {M}orse-{S}ard theorem in the case of
  infinite-dimensional manifolds.
\newblock {\em Proceedings of the American Mathematical Society},
  16(5):954--957, 1965.

\bibitem{Kushner2003}
Harold Kushner and George Yin.
\newblock {\em Stochastic Approximation and Recursive Algorithms and
  Applications}.
\newblock Springer-Verlag New York, 2003.

\bibitem{Kushner1978}
Harold~Joseph Kushner and Dean~S Clark.
\newblock {\em Stochastic Approximation Methods for Constrained and
  Unconstrained Systems}.
\newblock 1978.

\bibitem{Lee2013}
Hyung-Chun Lee and Jangwoon Lee.
\newblock A stochastic {G}alerkin method for stochastic control problems.
\newblock {\em Communications in Computational Physics}, 14(1):77--106, 2013.

\bibitem{Lei2018}
Jinlong Lei and Uday~V Shanbhag.
\newblock A randomized block proximal variable sample-size stochastic gradient
  method for composite nonconvex stochastic optimization.
\newblock {\em arXiv preprint arXiv:1808.02543}, 2018.

\bibitem{Ljung1977}
Lennart Ljung.
\newblock Analysis of recursive stochastic algorithms.
\newblock {\em IEEE Transactions on Automatic Control}, 22(4):551--575, 1977.

\bibitem{Martin2018}
Matthieu Martin, Sebastian Krumscheid, and Fabio Nobile.
\newblock Analysis of stochastic gradient methods for {PDE}-constrained optimal
  control problems with uncertain parameters.
\newblock Technical report, \'Ecole Polytechnique {MATHICSE} Institute of
  Mathematics, 2018.

\bibitem{Metivier2011}
Michel M{\'e}tivier.
\newblock {\em Semimartingales: A Course on Stochastic Processes}, volume~2.
\newblock Walter de Gruyter, 2011.

\bibitem{Nemirovski2009}
A.~Nemirovski, A.~Juditsky, G.~Lan, and A.~Shapiro.
\newblock Robust stochastic approximation approach to stochastic programming.
\newblock {\em SIAM Journal on Optimization}, 19(4):1574--1609, 2009.

\bibitem{Nixdorf1984}
Rainer Nixdorf.
\newblock An invariance principle for a finite dimensional stochastic
  approximation method in a {H}ilbert space.
\newblock {\em Journal of Multivariate Analysis}, 15:252--260, 1984.

\bibitem{Nouy2018}
Anthony Nouy and Florent Pled.
\newblock A multiscale method for semi-linear elliptic equations with localized
  uncertainties and non-linearities.
\newblock {\em ESAIM: Mathematical Modelling and Numerical Analysis},
  52(5):1763--1802, 2018.

\bibitem{Okada1984}
Nolio Okada et~al.
\newblock On the {B}anach-{S}aks property.
\newblock {\em Proceedings of the Japan Academy, Series A, Mathematical
  Sciences}, 60(7):246--248, 1984.

\bibitem{Pisier2016}
Gilles Pisier.
\newblock {\em Martingales in Banach spaces}, volume 155.
\newblock Cambridge University Press, 2016.

\bibitem{Reddi2016a}
Sashank~J Reddi, Suvrit Sra, Barnabas Poczos, and Alexander~J Smola.
\newblock Proximal stochastic methods for nonsmooth nonconvex finite-sum
  optimization.
\newblock In {\em Advances in Neural Information Processing Systems}, pages
  1145--1153, 2016.

\bibitem{Robbins1971}
H~Robbins and D~Siegmund.
\newblock A convergence theorem for non negative almost supermartingales and
  some applications.
\newblock In {\em Optimizing Methods in Statistics}, pages 233--257. Elsevier,
  1971.

\bibitem{Robbins1951}
Herbert Robbins and Sutton Monro.
\newblock A stochastic approximation method.
\newblock {\em The Annals of Mathematical Statistics}, 22(3):400--407, 1951.

\bibitem{Rosseel2012}
Eveline Rosseel and Garth Wells.
\newblock Optimal control with stochastic {PDE} constraints and uncertain
  controls.
\newblock {\em Computer Methods in Applied Mechanics and Engineering}, pages
  152--167, 2012.

\bibitem{Ruszczynski1983}
Andrzej Ruszczynski and Wojciech Syski.
\newblock Stochastic approximation method with gradient averaging for
  unconstrained problems.
\newblock {\em IEEE Transactions on Automatic Control}, 28(12):1097--1105,
  1983.

\bibitem{Shapiro1996}
Alexander Shapiro and Yorai Wardi.
\newblock Convergence analysis of stochastic algorithms.
\newblock {\em Mathematics of Operations Research}, 21(3):615--628, 1996.

\bibitem{Smale2000}
Stephen Smale.
\newblock An infinite dimensional version of {S}ard's theorem.
\newblock In {\em The Collected Papers of Stephen Smale: Volume 2}, pages
  529--534. World Scientific, 2000.

\bibitem{Tiesler2012}
Hanne Tiesler, Robert~M Kirby, Dongbin Xiu, and Tobias Preusser.
\newblock Stochastic collocation for optimal control problems with stochastic
  {PDE} constraints.
\newblock {\em SIAM Journal on Control and Optimization}, 50(5):2659--2682,
  2012.

\bibitem{Troeltzsch2009}
Fredi Tr\"oltzsch.
\newblock {\em Optimale Steuerung partieller Differentialgleichungen}.
\newblock Vieweg + Teubner, 2nd edition, 2009.

\bibitem{Uryasev1992}
Stanislav Uryasev.
\newblock A stochastic quasigradient algorithm with variable metric.
\newblock {\em Annals of Operations Research}, 39(1):251--267, 1992.

\bibitem{VanBarel2017}
Andreas Van~Barel and Stefan Vandewalle.
\newblock Robust optimization of {PDE} constrained systems using a multilevel
  {M}onte {C}arlo method.
\newblock In {\em SIAM/ASA Journal on Uncertainty Quantification}, volume~7,
  pages 174--202, 2019.

\bibitem{Venter1966}
J.H. Venter.
\newblock On {D}voretzky stochastic approximation theorems.
\newblock {\em The Annals of Mathematical Statistics}, 37:1534--1544, 1966.

\bibitem{Wardi1989}
Y~Wardi.
\newblock A stochastic algorithm using one sample point per iteration and
  diminishing stepsizes.
\newblock {\em Journal of Optimization Theory and Applications},
  61(3):473--485, 1989.

\bibitem{Williams1991}
David Williams.
\newblock {\em Probability with martingales}.
\newblock Cambridge University Press, 1991.

\bibitem{Yin1990}
G.~Yin and Y.M. Zhu.
\newblock On {$H$}-valued {R}obbins-{M}onro processes.
\newblock {\em Journal of Multivariate Analysis}, 34:116--140, 1990.

\end{thebibliography}
\end{document}